\documentclass{amsart}

\usepackage[english]{babel}
\usepackage[utf8]{inputenc}
\usepackage{amsmath}
\usepackage{amsfonts}
\usepackage{amssymb}
\usepackage{mathrsfs}
\usepackage{mathtools}
\usepackage{amsthm}
\usepackage[hidelinks]{hyperref}
\newtheorem{theorem}{Theorem}[section]
\newtheorem{lemma}[theorem]{Lemma}
\newtheorem{proposition}[theorem]{Proposition}
\newtheorem{corollary}[theorem]{Corollary}
\theoremstyle{definition}
\newtheorem{remark}[theorem]{Remark}

\title{Extreme values of geodesic periods on arithmetic hyperbolic surfaces}
\author{Bart Michels}
\thanks{Universit\'{e} Sorbonne Paris Nord, LAGA, CNRS, UMR 7539,  F-93430, Villetaneuse, France. Email: \texttt{michels@math.univ-paris13.fr}}

\begin{document}

\begin{abstract}Given a closed geodesic on a compact arithmetic hyperbolic surface, we show the existence of a sequence of Laplacian eigenfunctions whose integrals along the geodesic exhibit nontrivial growth. Via Waldspurger's formula we deduce a lower bound for central values of Rankin--Selberg $L$-functions of Maass forms times theta series associated to real quadratic fields.
\end{abstract}

\maketitle

\section{Introduction}

Let $X$ be a compact arithmetic hyperbolic surface and $- \Delta$ its Laplace operator, a self-adjoint operator on $L^{2}(X)$. The Hilbert space $L^{2}(X)$ admits an orthonormal basis $(\phi_{j})_{j \geq 0}$ of real-valued smooth simultaneous eigenfunctions for $- \Delta$ and the Hecke algebra, ordered by nondecreasing Laplace eigenvalue $\lambda_{j} \geq 0$.
It is known since Waldspurger \cite{waldspurger1985} (see also \cite{sarnak1995}) that the integrals of the $\phi_j$ over special orbits on $X$ are related to $L$-functions, making those \emph{compact periods} an interesting object of study. When $z \in X$ is a CM-point, associated to an imaginary quadratic field, it was shown in \cite{iwaniec1995} that the sequence of point evaluations $|\phi_j(z)|$ is unbounded: there exists a subsequence of $(\phi_j)$ on which
\begin{equation} \label{iwaniecsarnakgrowth}
|\phi_j(z)| \gg \sqrt{\log \log \lambda_{j}} \,.
\end{equation}
In \cite{milicevic2010}, this was strengthened to the existence of a subsequence on which
\begin{equation} \label{milicevicgrowth}
|\phi_j(z)| \gg \exp \left( C \sqrt {\frac{\log \lambda_{j}}{\log \log \lambda_{j}}} \right)
\end{equation}
for all $C < 1$ (and in fact, with the constant $C$ replaced by $1$ plus an error term, as in Theorem~\ref{maintheorem} below). In this context, the period formula relating the $\phi_j(z)$ to central values of $L$-functions is proven in \cite{zhang2001}. To put these lower bounds in perspective, the local Weyl law \cite{avakumovic1956, hormander1968} implies the convexity bound $|\phi_{j}(z)| \ll \lambda_{j}^{1/4}$, and implies that $|\phi_j(z)| \asymp 1$ `on average'. Note that \eqref{iwaniecsarnakgrowth} and \eqref{milicevicgrowth} imply a lower bound for the sup norms $\Vert \phi_j \Vert_\infty$. The \emph{sup norm conjecture} of Iwaniec and Sarnak \cite{iwaniec1995} asserts that $\Vert \phi_{j} \Vert_{\infty} \ll_{\epsilon} \lambda_{j}^{\epsilon}$ as $\lambda_j \to \infty$.

In this article, we prove a lower bound for integrals of eigenfunctions along closed geodesics, which are associated to \emph{real} quadratic fields. The lower bound is of a quality similar to \eqref{milicevicgrowth}, relative to the average value of such geodesic periods. To quantify what `average' means, let $\ell \subset X$ be a closed geodesic, and denote by $\mathscr{P}_{\ell}(\phi_{j})$ the integral of $\phi_{j}$ along $\ell$. In \cite{zelditch1992} it is shown that
\[
\sum_{\lambda_{j} \leq \lambda} \mathscr{P}_{\ell}(\phi_{j})^{2} = C_{\ell} \lambda^{1/2} + O(1) \qquad (\lambda \to \infty) \,,
\]
for some $C_{\ell} > 0$. This implies that $|\mathscr{P}_{\ell}(\phi_{j})| \ll 1$, which is the convexity bound in this setting.  By the Weyl law, the above sum consists of $\asymp \lambda$ terms, so that $\mathscr{P}_{\ell}(\phi_{j})^{2} \asymp \lambda^{-1/2}$ `on average'.
In relation to a Waldspurger-type formula, the Lindel\"{o}f hypothesis for certain $L$-functions leads to the conjecture that $\lambda_{j}^{1/4}|\mathscr{P}_{\ell}(\phi_{j})| \ll_{\epsilon} \lambda_{j}^{\epsilon}$ \cite{reznikov2015}. A geodesic period is `large' when the normalized period $\lambda_{j}^{1/4} | \mathscr{P}_{\ell}(\phi_{j})|$ is substantially larger than $1$. We can now state our main theorem:

\begin{theorem} \label{maintheorem} Let $R$ be an Eichler order of square-free level in an indefinite quaternion division algebra over $\mathbb{Q}$. Let $\Gamma \subset \operatorname{PSL}_{2}(\mathbb{R})$ be the corresponding co-compact arithmetic lattice. Let $(\phi_{j})$ be an orthonormal basis of $L^{2}(\Gamma \backslash \mathfrak{h})$ consisting of Laplace--Hecke eigenfunctions (see \S \ref{notationmaassforms}) with Laplacian eigenvalues $\lambda_{j} \geq 0$. Let $\ell \subset \Gamma \backslash \mathfrak{h}$ be a closed geodesic. Then there exists $C > 0$ such that
\begin{equation}\label{maintheoremeqn}
\max_{|\sqrt{\lambda}_{j} - \sqrt{\lambda}| \leq C} \lambda_{j}^{1 / 4} \left\vert \mathscr{P}_{\ell}(\phi_{j}) \right\vert \geq \exp \left( \frac{1}{2} \sqrt{\frac{\log \lambda}{\log \log \lambda}} \left( 1 + O \left( \frac{\log \log \log \lambda}{\log \log \lambda} \right) \right) \right)
\end{equation}
as $\lambda \to \infty$.
\end{theorem}

It is not a real restriction to assume that $R$ is an Eichler order of square-free level; see Remark~\ref{remarkarbitraryorder}. Geodesic periods are related to central values of Rankin--Selberg $L$-functions via a period formula; see Remark~\ref{remarkpopaformula} for a precise statement. As a corollary, we obtain:

\begin{corollary} \label{corollaryLfunction} Let $N > 1$ be a square-free integer, $F$ a real quadratic number field of square-free discriminant coprime to $N$. Assume that $N$ has an even number of prime divisors and that they are all inert in $F$.
Let $(f_{j})_{j \geq 1}$ be an orthonormal basis of the space of even Maass newforms of level $\Gamma_{0}(N)$ with eigenvalues $\lambda_{j} > 0$. There exists $C > 0$ such that
\begin{equation} \label{lfunctiongrowth}
\max_{|\sqrt{\lambda}_{j} - \sqrt{\lambda}| \leq C} |L(1/2, f_{j}) L(1/2, f_{j} \times \omega_{F}) | \geq \exp \left( \sqrt{\frac{\log \lambda}{\log \log \lambda}} \left( 1 + o(1) \right) \right) \,,
\end{equation}
where $\omega_{F}$ is the Dirichlet character associated to the field extension $F/ \mathbb{Q}$ by class field theory, and the implicit constant depends on $N$ and $F$.
\end{corollary}

We remark that the lower bound in \eqref{lfunctiongrowth} is of a quality similar to up-to-date lower bounds for extreme central values in various families of $L$-functions, obtained using the resonance method of Soundararajan \cite{soundararajan2008} and Hilberdink \cite{hilberdink2009}. Typically, one considers the Riemann zeta function on the critical line, the central value of Dirichlet $L$-functions of characters to large moduli or the central value of $L$-functions of modular forms of large weight. We refer to \cite{blomer2020, bondarenko2017, delabreteche2019} for an account of recent results.

In \cite{farmer2007} the authors explore how far such results are from being optimal, obtaining conjectures based on random models for $L$-functions: when $\mathcal{F}$ is a suitable family of $L$-functions, there exists $B > 0$ such that the maximum of $|L(1/2)|$ when $L$ runs through the $L$-functions in $\mathcal{F}$ with analytic conductor at most $D$, is
\[ \exp \left( (B + o(1)) \sqrt{\log D \log \log D} \right) \]
as $D \to \infty$. By Waldspurger's formula, this would imply that \eqref{maintheoremeqn} is true with the exponent in the RHS replaced by $B \sqrt{\log \lambda \log \log \lambda}$ for some $B > 0$, if the $\max$ in the LHS runs over $\lambda_j \in [0, \lambda]$.

The proof of Theorem~\ref{maintheorem} uses the celebrated amplification method of Iwaniec and Sarnak \cite{iwaniec1995}, whose application to extreme values is reminiscent of the resonance method. The amplification method consists of comparing two trace formulas with an appropriate choice of adelic test functions. It exploits the presence of the Hecke operators, and the fact that many Hecke returns fix a given set (a CM-point or a closed geodesic). This is the approach in \cite{iwaniec1995} and \cite{milicevic2010}, which establish \eqref{iwaniecsarnakgrowth} and \eqref{milicevicgrowth} respectively, and we follow it as well. We send the reader to \S \ref{sectionoutline} for an outline.
The amplification method can be thought of as a way to quantify the phenomenon that eigenfunctions have strong concentration properties at points or geodesics that are fixed by many symmetries, a phenomenon that is also observed for zonal spherical harmonics and Gaussian beams on round spheres.

The method stands in contrast to another strategy to prove the existence of large values of compact periods, which uses functoriality and a vanishing property to show that a certain sparse subsequence of $(\phi_{j})$ must attain large values on a distinguished set; see e.g.\ \cite{rudnick1994}. In some cases, the eigenfunctions one obtains by the amplification method are also related to functoriality (see \cite{brumley2020, milicevic2011} for a discussion of the relation between the two methods). But on hyperbolic surfaces, that does not seem to be the case.

\subsection{The random wave conjecture and large values}

\label{RWC}

The fact that on arithmetic hyperbolic surfaces the sequence $\Vert \phi_{j} \Vert_{\infty}$ is unbounded, can, as remarked above, be explained through random models for $L$-functions, but also by random behavior of Laplacian eigenfunctions. In this subsection and the next, we explore how much growth can be explained by randomness of eigenfunctions, and we discuss some recent developments and related results.

Unless otherwise stated, let $X$ be a compact Riemannian surface with negative sectional curvature. (Some of the results and conjectures stated will apply to arbitrary curvature or arbitrary dimension, but generality is not our objective.) Let $(\phi_j)$ be an orthonormal basis of $L^2(X)$ consisting of Laplacian eigenfunctions. (If $X$ is a noncompact finite-volume quotient of $\mathfrak{h}$, we take it to be an orthonormal basis of the space of Maass cusp forms.)
Because $X$ is compact and of negative curvature, its geodesic flow is ergodic, and the \emph{random wave conjecture} of Berry~\cite{berry1977} predicts that Laplacian eigenfunctions of large eigenvalue should show Gaussian random behavior. Berry's conjecture has been tested numerically in \cite{aurich1991, aurich1993} for certain compact hyperbolic surfaces of genus $2$, and in \cite{hejhal1992} for the modular surface $\operatorname{PSL}_{2} (\mathbb{Z}) \backslash \mathfrak{h}$ (which, although non-compact, still has ergodic geodesic flow). In this last article, the authors propose the following mathematical interpretation of Berry's conjecture: Equip $X$ with the normalized Riemannian probability measure $\operatorname{Vol}(X)^{-1} d \operatorname{Vol}$. Then the sequence of eigenfunctions $(\phi_{j})$, considered as random variables on $X$, converges weakly to the normal distribution with mean $0$ and variance $\operatorname{Vol}(X)^{-1}$.
The random wave conjecture has been the subject of recent work \cite{abert2018} where it has been interpreted in terms of Benjamini--Schramm convergence of manifolds. In \cite[Theorem 4]{abert2018} is also shown that this version of the conjecture implies the quantum unique ergodicity conjecture (QUE) of Rudnick and Sarnak~\cite{rudnick1994} which states that the probability measures $\phi_{j}^{2} d \operatorname{Vol}$ converge weakly to the uniform measure $\operatorname{Vol}(X)^{-1} d \operatorname{Vol}$.

Even though Berry's conjecture remains widely open in any interpretation, it is natural to wonder about the finer statistical properties of Laplacian eigenfunctions. For example, when the random wave conjecture is regarded as an analogue of the central limit theorem, one could ask what the analogue of the law of the iterated logarithm should be. That is, what can be said about the growth of the sup norm $\Vert \phi_{j} \Vert_{\infty}$ as $j \to \infty$? In order to obtain predictions for this type of questions, it is common to study \emph{random wave models}, probability measures on a space of functions. Following \cite[\S 6]{zelditch2005}, we may model an eigenfunction $\phi$ of large eigenvalue $\lambda$ by the random variable
\[ \frac{1}{(\sum_{j} c_{j}^{2})^{1/2}} \sum_{j} c_{j} \phi_{j} \]
where the $c_{j}$ are i.i.d.\ standard Gaussians and the sums run over the set $\{ j : \sqrt{\lambda}_{j} \in [\sqrt{\lambda}, \sqrt{\lambda} + 1] \}$. One thus expects the sup norm $\Vert \phi \Vert_{\infty}$ to generically have the same behavior as the sup norm of random Fourier series studied in \cite[\S 6]{kahane1985}, \cite[Chapter IV]{salem1954}. That is, that $\Vert \phi_{j} \Vert_{\infty} \asymp \sqrt{\log \lambda_{j}}$ for a subsequence of $(\phi_{j})$ \cite{hejhal1992, sarnak1995}. This is compatible with the sup norm conjecture and with \eqref{iwaniecsarnakgrowth}. As for \eqref{milicevicgrowth}, note that the RHS is eventually smaller than any power of $\lambda_{j}$, and eventually larger than any power of $\log \lambda_{j}$. Such large sup norms, while not in contradiction with the sup norm conjecture nor the random wave conjecture, should be considered exceptional and are presumably specific to the case of arithmetic manifolds: in both articles, one uses the recurrence properties of CM-points under Hecke operators. Considering these results, it is somewhat surprising that despite the presence of Hecke operators, numerical evidence indicates that the random wave conjecture does hold for arithmetic hyperbolic surfaces.

\subsection{Restricted randomness and random restrictions}

One can think of at least two ways to generalize questions about random behavior of the $\phi_{j}$. On the one hand, one can fix a `thin' set $S$, such as a finite set or a geodesic segment, equipped with its induced Riemannian metric and volume element $\mu_{S}$, and ask about the value distribution of the restrictions $\phi_{j}|_{S}$. It is clear that any result must be very sensitive to the nature of the set $S$. We illustrate this with an example. Quantum (unique) ergodic restriction, Q(U)ER, is the question of whether a density one subsequence (resp.\ the full sequence) of  the restrictions $\phi_{j}^{2}|_{S}$ equidistribute. It was first studied in \cite{toth2012}. When $X$ is the modular surface, the sequence $(\phi_{j})$ of Maass cusp forms can be chosen to consist of odd and even cusp forms. The odd ones vanish on the split geodesic $S = [i, i \infty)$, and one expects QUER to hold for the restrictions to $S$ of the even Maass forms, but where the limiting measure is $2 \mu_{S}$ instead of $\mu_{S}$ \cite{young2016}. We mention also \cite{young2018}, where the analogue of QER for Eisenstein series is found to have a negative answer on certain divergent geodesics on the modular surface. In the other direction, QER does hold for generic geodesics on a compact hyperbolic surface \cite{dyatlov2013, toth2013}.

The other point of view consists of randomizing the thin set. Given an eigenfunction and a geodesic segment $L$, we may form the integral
\[ \mathscr{P}_{L} (\phi_{j}) := \int_{L} \phi_{j} \,. \]
If we fix a real number $l > 0$, the set of geodesic segments of length $l$ is parametrized by the unit tangent bundle $S^{1} X$. Every eigenfunction $\phi_{j}$ gives rise to a smooth function $\widetilde{\phi}_{j}$ on $S^{1} X$, by integrating over the segment corresponding to a point of $S^{1} X$. It is known that the sequence $\Vert \widetilde{\phi}_{j} \Vert_{\infty}$ is bounded as $\lambda_{j} \to \infty$ (by a constant depending on $X$ and $l$) \cite{chen2015}. In particular, when $\ell$ is a closed geodesic, we recover the convexity bound $|\mathscr{P}_{\ell}(\phi_{j})| \ll 1$, which is sharp when $X$ is the round $2$-sphere. Theorem~\ref{maintheorem} can be viewed as producing large values of the functions $\widetilde{\phi}_{j}$ for suitable $l$.

We remark that, in order to obtain heuristics for large values of $\tilde{\phi}_{j}$, to apply random wave models as in \S \ref{RWC} would be nonsensical, because Berry's conjecture is fundamentally a statement about the value distribution of eigenfunctions, and does not take into account the relative position of points. That is, the random wave models are not designed to be integrated over sets of measure zero. Instead, it would be interesting to develop a conjecture analogous to Berry's for the value distribution of $\tilde{\phi}_{j}$ as $\lambda_{j} \to \infty$. A related type of result is the central limit theorem for geodesic flows \cite{ratner1973, sinai1960}: for a fixed smooth function $\phi \in C^{\infty}(X)$, one determines the limiting distribution of the (normalized) integral of $\phi$ along a random geodesic segment of growing length.

\subsection{Organization of the paper}

\label{sectionoutline}

The proof of Theorem~\ref{maintheorem} goes by a comparison of trace formulas. One is the Selberg trace formula; the other is a relative (pre)-trace formula involving geodesic periods. In \S \ref{twotraceformulas}, we derive both formulas and polish the first by identifying a main term and bounding the error term. Here, a novelty is the estimation of the contribution of hyperbolic classes in the Selberg trace formula, which is an unavoidable issue given the nature of the problem; see Remark~\ref{remarkdifferentconvolutionoperator}. In the relative pre-trace formula, we similarly identify a main term, understanding which involves an arithmetic (\S \ref{countingstabilizers}) and an analytic question (\S \ref{maintermanalysis}). Bounding the error term leads to a Diophantine problem (\S \ref{countingalmostheckereturns}) and the analytic problem of bounding orbital integrals (\S \ref{orbitalintegrals}). In \S \ref{sectionamplification}, we combine the work and consider an arbitrary amplifier, obtaining an asymptotic estimate for short spectral sums. In \S \ref{sectionoptimization} we optimize the amplifier, proving Theorem~\ref{maintheorem}. In \S \ref{sectioncorollary} we adapt the method to prove Corollary~\ref{corollaryLfunction}.

\subsection{Acknowledgements}

I thank Farrell Brumley for giving me this problem, and for helpful suggestions. I also thank Nicolas Bergeron and \'{E}tienne Le Masson for useful conversations surrounding the topic of this paper.

\section{Notation}
\label{notationsection}

\subsection{Lie groups}
\label{distancesequivalent}

Let $G = \operatorname{PSL}_{2}(\mathbb{R})$ and define the subgroups $K = \operatorname{PSO}_{2}(\mathbb{R})$, $A = \left\{ \begin{psmallmatrix} * & 0 \\ 0 & * \end{psmallmatrix} \right\}$ and $N = \left\{ \begin{psmallmatrix} 1 & * \\ 0 & 1 \end{psmallmatrix} \right\}$. Fix parametrizations $a : \mathbb{R} \to A$, $k : \mathbb{R} / 2 \pi \mathbb{Z} \to K$ defined by
\[
a(t) = \begin{pmatrix}
e^{t/2} & 0 \\
0 & e^{-t/2}
\end{pmatrix} \quad, \quad k(\theta) = \begin{pmatrix}
\cos(\theta / 2) & \sin(\theta / 2) \\
- \sin(\theta / 2) & \cos(\theta / 2)
\end{pmatrix} \,.
\]
The action of $G$ on the hyperbolic plane $\mathfrak{h}$ induces a map $G \to \mathfrak{h} : g \mapsto gi$ and a diffeomorphism $N \times A \cong \mathfrak{h}$. Denote by $dg$ be the (bi-invariant) Haar measure on $G \cong \mathfrak{h} \times K$ that is the product of the hyperbolic measure with the Haar measure of mass $1$ on $K$. Fix any Riemannian metric on $G$, and denote the Riemannian distance on $G$ by $d = d_{G}$.

We will often use the fact that any two Riemannian distances are locally equivalent, or more generally the following fact: When $M$ is a smooth connected manifold, $N \subset M$ a connected submanifold, each equipped with a Riemannian metric, then for every compact subset $L \subset N$ and $x, y \in L$ we have $d_{N}(x, y) \asymp_{L} d_{M}(x, y)$.
In particular, taking $M = N = G$ and pulling back the metric on $G$ on the left or the right by $g \in G$, we have for every compact $L \subset G$ and $x, y \in L$ that $d(gx, gy) \asymp_{g, L} d(xg, yg) \asymp_{g, L} d(x, y)$.
Taking $N = \operatorname{SL}_{2}(\mathbb{R})$ equipped with any Riemannian metric and $M = M_{2}(\mathbb{R})$ with the Euclidean metric w.r.t.\ the standard basis, we find that for $L \subset \operatorname{SL}_{2}(\mathbb{R})$ compact and $x, y \in L$, $d_{\operatorname{SL}_{2}}(x, y) \asymp_{L} \Vert x - y \Vert_{2}$.
We also have $d_{G}(x, y) \asymp_{L} \min(d_{\operatorname{SL}_{2}}(x, y), d_{\operatorname{SL}_{2}}(x, -y))$.

\subsection{Arithmetic quotients}

\label{quaternionalgebranotation}

Let $B$ be a quaternion division algebra over $\mathbb{Q}$ which is split at $\infty$. That is, there exists an isomorphism of $\mathbb{R}$-algebras $\rho : B \otimes_{\mathbb{Q}} \mathbb{R} \overset{\sim}{\rightarrow} M_{2}(\mathbb{R})$. We view $B$ as a subset of $B \otimes_{\mathbb{Q}} \mathbb{R}$ via the natural embedding. Denote the projection $\operatorname{GL}_{2}^{+}(\mathbb{R}) \to \operatorname{PSL}_{2}(\mathbb{R}) = G$ by $g \mapsto \overline{g}$, and denote $\overline{\rho}(g) = \overline{\rho(g)}$ for an element $g \in (B \otimes_{\mathbb{Q}} \mathbb{R})^{+}$ of positive reduced (quaternion) norm. We will not always distinguish between $\eta \in B^{+}$ and its image $\overline{\rho}(\eta) \in G$. Let $R \subset B$ be a $\mathbb{Z}$-order. For $n \in \mathbb{N}_{> 0}$, denote by $R(n) \subset B^{+}$ the set of elements of reduced norm equal to $n$, denote $R^{1} = R(1)$ and define $\Gamma = \overline{\rho}(R^{1}) \subset G$. It is well known that $\Gamma$ is a lattice in $G$, and that the quotient $\Gamma \backslash \mathfrak{h}$ is compact (see e.g.\ \cite[IV: Th\'{e}or\`{e}me 1.1.]{vigneras1980}). We call $\Gamma$ an arithmetic lattice.

Let $\Vert \cdot \Vert$ be the operator norm on $M_{2}(\mathbb{R}) \cong B \otimes_{\mathbb{Q}} \mathbb{R}$. It is invariant under quaternion conjugation.

\subsection{Closed geodesics}

\label{notationclosedgeodesics}

A \emph{geodesic} in $\mathfrak{h}$ is the image of a maximal geodesic curve. The image of $A$ in $\mathfrak{h}$ is the geodesic joining $0$ with $\infty$. Because $G$ acts transitively on the unit tangent bundle of $\mathfrak{h}$, it acts transitively on geodesics. If $L$ is the geodesic that is the image of $gA$ in $\mathfrak{h}$, its set-wise stabilizer is the normalizer $N_{G}(gAg^{-1})$, in which $gAg^{-1}$ has index equal to $2$.

If $\Gamma \subset G$ is a discrete subgroup, a geodesic $L = gAi \subset \mathfrak{h}$ projects to a smooth curve in the Riemannian orbifold $\Gamma \backslash \mathfrak{h}$. It has a periodic image $\ell \subset \Gamma \backslash \mathfrak{h}$ precisely when there exists $\gamma \in (\Gamma \cap gAg^{-1}) - \{ 1 \}$ which stabilizes $L$. We then call $\ell$ a closed geodesic and denote $\Gamma_{L} = \operatorname{Stab}_{\Gamma}(L) \cap gAg^{-1}$. It is a lattice in $gAg^{-1}$ and we have $[\operatorname{Stab}_{\Gamma}(L) : \Gamma_{L}] \in \{ 1, 2 \}$. When this index equals $1$, $\ell$ is called a reciprocal geodesic; see \cite{sarnak2007}. When $\ell$ is a closed geodesic and $\phi \in C^{\infty}(\Gamma \backslash \mathfrak{h})$, define the period of $\phi$ along $\ell$ as the line integral
\begin{equation} \label{definitionarithmeticperiodintegral}
\mathscr{P}_{\ell}(\phi) := \int_{\Gamma_{L} \backslash L} \phi \,.
\end{equation}

When $\Gamma$ is an arithmetic lattice as in \S \ref{quaternionalgebranotation}, the closed geodesics can be characterized as follows. When $F \subset B$ is a real quadratic number field, the $\mathbb{R}$-algebra $F \otimes_{\mathbb{Q}} \mathbb{R} \subset B \otimes_{\mathbb{Q}} \mathbb{R}$ is isomorphic to $\mathbb{R} \times \mathbb{R}$, hence its image under $\rho$ is conjugate in $M_{2}(\mathbb{R})$ to the algebra of diagonal matrices. Thus the group $\bar{\rho}((F \otimes_{\mathbb{Q}} \mathbb{R})^{1})$ of matrices of determinant $\pm 1$ equals $g A g^{-1}$ for a unique $g \in G / N_{G}(A)$. One has that $g A$ projects to a closed geodesic in $\Gamma \backslash \mathfrak{h}$ and every closed geodesic is obtained exactly once in this way: we obtain a bijection $F \mapsto L_{F}$ between real quadratic number fields inside $B$ and geodesics in $\mathfrak{h}$ that become closed in $\Gamma \backslash \mathfrak{h}$. (In particular, whether a geodesic becomes closed in $\Gamma \backslash \mathfrak{h}$ depends only on $B$ and $\rho$.) It induces a bijection between real quadratic $F \subset B$ modulo conjugation by $R^{1}$, and closed geodesics $\ell \subset \Gamma \backslash \mathfrak{h}$.

When $F$ is such a real quadratic field, we write $R_{F} = R \cap F$. It is an order in $F$ and one has that $\Gamma_{L_{F}} = \bar{\rho}(R_{F}^{1})$, whereas $\operatorname{Stab}_{\Gamma}(L_{F})$ equals the normalizer $\overline{\rho}(N_{R^{1}}(F))$.
More generally, one has $\operatorname{Stab}_{B^{+}}(L_{F}) = N_{B^{+}}(F)$.
Let $\omega \in B^{\times}$ be a Skolem--Noether element w.r.t.\ $F$. That is, conjugation by $\omega$ leaves $F$ invariant and induces the non-trivial automorphism of $F$. Then $N_{B^{+}}(F) = F^{+} \sqcup (\omega F)^{+}$.

\begin{remark}
Closed geodesics are naturally grouped in packets indexed by a class group, and the above bijection gives the geodesic corresponding to the identity of the group. When $R$ is an Eichler order of square-free level, packets can be described without use of the adelic language as follows: Let $F \subset B$ be a real quadratic field, and $I$ an invertible fractional $R_{F}$-ideal. The right $R$-ideal $IR$ is principal and generated by an element $a \in B^{+}$. To $I$ we associate the $R^{1}$-conjugacy class of the field $a^{-1} F a$ (i.e., a closed geodesic). The map obtained in this way factors through the narrow class group of $R_{F}$ to an injective map.
\end{remark}

\begin{remark}
When a closed geodesic is viewed as a subset of $\Gamma \backslash \mathfrak{h}$, it may seem natural to define $\mathscr{P}_{\ell}(\phi_{j})$ instead as an integral over $\operatorname{Stab}_{\Gamma}(L) \backslash L$, resulting in a period that is half as large when $L$ is a reciprocal geodesic. However, the definition \eqref{definitionarithmeticperiodintegral} is closer to the notion of an adelic period, and appears naturally in a period formula, which we now state.
\end{remark}

\begin{remark} \label{remarkpopaformula}
Geodesic periods are related to central values of Rankin--Selberg $L$-functions, as follows: \cite[Theorem 5.4.1]{popa2006} Let $N \geq 1$ be a square-free integer, and let $F$ be a real quadratic field of square-free discriminant $d_{F}$ coprime to $N$. Assume that the number of primes dividing $N$ which are inert in $F$ is even. Let $B$ be a quaternion algebra over $\mathbb{Q}$ which is ramified exactly at the primes dividing $N$ which are inert in $F$. Fix an embedding $\iota : F \hookrightarrow B$. Let $R \subset B$ be an Eichler order of level $N / \Delta_{B}$ containing the full ring of integers $\iota(\mathcal{O}_{F})$. (The freedom in the choice of the Eichler order $R$ is explained in \cite[Remark 5.3.3]{popa2006}.) Let $\Gamma \subset \operatorname{PSL}_{2}(\mathbb{R})$ be the corresponding lattice. Let $h_{F}^{+}$ be the narrow class number of $F$. To $F$ corresponds a packet $\Lambda_{d_{F}}$ consisting of $h_{F}^{+}$ closed geodesics in $\Gamma \backslash \mathfrak h$. Let $\chi_0$ be the trivial character of the narrow class group of $F$ and $\pi_{\chi_0}$ the associated automorphic representation of $\operatorname{GL}_{2}(\mathbb{A}_{\mathbb Q})$ given by the theta correspondence. Let $f$ be an even Maass newform of level $\Gamma_{0}(N)$. Let $\pi_{f}$ be the automorphic representation of $\operatorname{GL}_{2}(\mathbb{A}_{\mathbb{Q}})$ it generates. Let $\pi_{f}^{\operatorname{JL}}$ be the corresponding representation of $B^\times(\mathbb A_{\mathbb Q})$ given by the Jacquet--Langlands correspondence and take a newvector in $\pi_{f}^{\operatorname{JL}}$, which corresponds to an $L^{2}$-normalized Hecke--Laplace eigenfunction $\phi$ on $\Gamma \backslash \mathfrak{h}$. Then
\begin{equation} \label{popaformula}
\frac{\Lambda(1/2, \pi_{f} \times \pi_{\chi_0})}{\Lambda(1, \pi_{f}, \operatorname{Ad})} = 2 d_{F}^{-1/2} \prod_{p \mid \Delta_{B}} \frac{p + 1}{p - 1} \cdot \left\vert \sum_{\ell \in \Lambda_{d_{F}}} \mathscr{P}_{\ell}(\phi) \right\vert^{2} \,,
\end{equation}
where the $\Lambda$'s in the LHS are completed Rankin--Selberg $L$-functions.

Note that: (1) The assumption that there is an even number of prime divisors of $N$ that are inert in $F$, implies that the $\epsilon$-factor in the functional equation for $\Lambda(s, \pi_{f} \times \pi_{\chi_0})$ is $1$, so that there is no forced vanishing of the central value; (2) If there is \emph{at least one} such prime divisor, $B$ is a division algebra; (3) If \emph{all} prime divisors of $N$ are inert in $F$, $R$ is a maximal order. In Corollary~\ref{corollaryLfunction} we impose all three conditions.
\end{remark}

\begin{remark} \label{remarkarbitraryorder}
If in Theorem~\ref{maintheorem} we drop the requirement that the subsequence of the $(\phi_{j})$ consists of Hecke eigenfunctions, the statement holds for an arbitrary order $R$. Indeed, one may reduce to the case of a maximal order, as follows: if $R$ is any order, let $R'$ be a maximal order containing $R$, and $\Gamma'$ the corresponding lattice. When $\phi'$ is an $L^{2}$-normalized eigenfunction for $\Gamma' \backslash \mathfrak{h}$, then $\phi := [\Gamma' : \Gamma]^{-1/2} \phi'$ is one for the cover $\Gamma \backslash \mathfrak{h}$, and $\phi$ and $\phi'$ have the same Laplacian eigenvalue. The periods $\mathscr{P}(\phi)$ and $\mathscr{P}' (\phi')$ w.r.t.\ $\Gamma$ and $\Gamma'$ are related by
\[ \mathscr{P}(\phi) = [\Gamma' : \Gamma]^{-1/2} [\Gamma'_{L} : \Gamma_{L}] \mathscr{P}' (\phi') \,.\]
Thus large geodesic periods on $\Gamma' \backslash \mathfrak{h}$ produce large geodesic periods on $\Gamma \backslash \mathfrak{h}$. The reason we do restrict to Eichler orders of square-free level, is that the Hecke theory for such orders is well documented. 
\end{remark}

\subsection{Hecke operators}

Assume from now on that $R$ is an Eichler order of square-free level. Let $\Delta_{R}$ be its discriminant and $\Gamma = \overline{\rho}(R^{1})$. For $n \in \mathbb{N}_{> 0}$, the quotients $R(1) \backslash R(n)$ are finite and we define the Hecke operators $T_{n}$, which act on $L^{2}(\Gamma \backslash \mathfrak{h})$, by
\[ (T_{n} f)(z) = \sum_{\eta \in R(1) \backslash R(n)} f(\eta z) \,. \]
The $T_{n}$ are self-adjoint operators and commute with the Laplacian $\Delta$ on smooth functions. Because $R$ is an Eichler order of square-free level, when $(mn, \Delta_{R}) = 1$ we have the relations
\begin{equation} \label{Heckerecurrence}
T_{m} T_{n} = T_{n} T_{m} = \sum_{d \mid m, n} d \cdot T_{mn / d^{2}}
\end{equation}
(see \cite[\S III.7]{eichler1965}). Note that the Hecke operator $T_{n}$ is often normalized by multiplying the above sum by a factor $1/ \sqrt{n}$, which we don't do here.

\subsection{Convolution operators}

\label{sectionconvolutionoperators}

In order to extract short spectral sums from the spectral side of the pre-trace formula, we will make use of convolution operators that are approximate spectral projectors, and whose Harish-Chandra transform satisfies certain positivity properties. Their existence is guaranteed by Proposition~\ref{choiceconvolution}.

For $s \in \mathbb{C}$, define the spherical function $\varphi_{s} \in C^{\infty}(K \backslash G / K)$ by
\[ \varphi_{s}(g) = \int_{K} e^{(1 / 2 + s) H(kg)} dk \,, \]
where $H : G \to \mathbb{R}$ is defined by requiring that $g \in N a(H(g)) K$.
Define the Harish-Chandra transform
\[ \widehat{k}(s) = \int_{G} k(g) \varphi_{s}(g) dg \,, \]
which is an entire function of finite type whose type depends on the support of $k$.

\begin{proposition} \label{choiceconvolution} There exists a family $(k_{\nu})_{\nu \geq 0}$ of bi-$K$-invariant smooth functions on $G$ satisfying:
\begin{enumerate}
\item There exists $R > 0$ such that $k_{\nu}$ is supported in the ball $B(e, R)$ for all $\nu$;
\item $\widehat{k}_{\nu}(s) \in \mathbb{R}_{\geq 0}$ when $s \in \mathbb{R} \cup i \mathbb{R}$;
\item $\widehat{k}_{\nu}(i r) \geq 1$ for $|r - \nu| \leq 1$;
\item $\widehat{k}_{\nu}(i r) \ll_{N} (1 + |\nu - r|)^{-N}$ uniformly for $\nu \geq 0$ and $r \in \mathbb{R}_{\geq 0} \cup [-i/2, i/2]$;
\item $k_{\nu}(e) \asymp \nu$ as $\nu \to \infty$.
\end{enumerate}
\end{proposition}

\begin{proof} Take any real-valued $k \in C_{0}^{\infty}(K \backslash G / K)$. The Harish-Chandra transform $\widehat{k}(s)$ is clearly real when $s \in \mathbb{R}$, and in view of the functional equation $\varphi_{-s} = \varphi_{s}$, also when $s \in i \mathbb{R}$. If we choose $k$ nonnegative and not identically zero, we have $\widehat{k}(s) > 0$ for $s \in \mathbb{R}$; in particular for $s = 0$. Let $k_{(1)} = k * k$. Then $\widehat{k}_{(1)} (s) = \widehat{k}(s)^{2} \geq 0$ for $s \in \mathbb{R} \cup i \mathbb{R}$. There exists $\delta > 0$ such that $\widehat{k}_{(1)} (s) > \widehat{k}_{(1)} (0) / 2$ for $|s| \leq \delta$. Let $\widehat{k}_{(2)} (s) = 2 \widehat{k}_{(1)} (\delta s) / \widehat{k}_{(1)} (0)$. Then $\widehat{k}_{(2)} (s) \geq 1$ for $|s| \leq 1$. Let $k_{(2)}$ be its inverse Harish-Chandra transform. Define 
\[ \widehat{k}_{\nu}(s) = \widehat{k}_{(2)} (i \nu + s) + \widehat{k}_{(2)} (i \nu - s) \,.\]
Then $\widehat{k}_{\nu}$ is entire, even, of exponential type (at most) the exponential type of $\widehat{k}_{(2)}$,
so it is the Harish-Chandra transform of a $k_{\nu} \in C_{0}^{\infty}(K \backslash G / K)$ whose support is bounded independently of $\nu$. Conditions 1.\ and 2.\ are now satisfied. When $|r - \nu| \leq 1$, we have
\begin{align*}
\widehat{k}_{\nu}(ir) & = \widehat{k}_{(2)} (i \nu + ir) + \widehat{k}_{(2)} (i \nu - ir) \\
& \geq 0 + 1 \,,
\end{align*}
which implies condition 3. Being the Fourier transform of the Abel transform of $k_{(2)}$, $\widehat{k}_{(2)}$ is of rapid decay on vertical strips of $\mathbb{C}$. Thus when $\operatorname{Re}(r) \geq 0$ and $\operatorname{Im}(r)$ remains bounded,
\begin{align*}
\widehat{k}_{\nu}(ir) & \ll_{N} (1 + |\nu + r|)^{-N} + (1 + |\nu - r|)^{-N} \\
& \ll (1 + |\nu - r|)^{-N} \,,
\end{align*}
which is 4. For 5., using the inverse Harish-Chandra transform we have
\begin{equation} \label{inverseHarishChandra}
k_{\nu}(g) = \int_{0}^{\infty} \widehat{k}_{\nu}(ir) \varphi_{ir}(g) \beta(r) dr \,,
\end{equation}
where $\beta(r) = \frac{1}{2 \pi} \tanh(\pi r) \cdot r$ is the Plancherel density. When $g = e$, $\varphi_{ir}(g) = 1$ so that
\[ 
k_{\nu}(e) \geq \int_{\nu}^{\nu + 1} \widehat{k}_{\nu}(ir) \beta(r) dr \gg \nu
\]
on the one hand, and
\begin{align*}
k_{\nu}(e) & \ll \int_{0}^{2 \nu} \widehat{k}_{\nu}(ir) \nu dr + \int_{2 \nu}^{\infty} \frac{r}{(r - \nu)^{3}} dr \\
& \ll \nu \cdot \Vert \widehat{k}_{\nu} \Vert_{L^{1}(i \mathbb{R})} + \nu^{-1} \\
& \ll \nu
\end{align*}
on the other.
\end{proof}

Throughout, we fix a family $(k_{\nu})$ as in Proposition~\ref{choiceconvolution}.

\subsection{Eigenfunctions}

\label{notationmaassforms}

Let $(\phi_{j})_{j \geq 0}$ be an orthonormal basis of $L^{2}(\Gamma \backslash \mathfrak{h})$ consisting of simultaneous eigenfunctions for $\Delta$ and the Hecke operators $T_{n}$ for $(n, \Delta_{R}) = 1$, ordered by nondecreasing Laplacian eigenvalue $\lambda_{j} \geq 0$:
\[ (\Delta + \lambda_{j}) \phi_{j} = 0 \,.\]
Write the $n$th Hecke eigenvalue of $\phi_{j}$ as $\widehat{T}_{n}(\phi_{j})$.
For each $j$, write $\lambda_{j} = \frac{1}{4} + \nu_{j}^{2}$, where $\nu_{j} \in \mathbb{R}_{\geq 0} \cup [0, i/2]$ is called the spectral parameter of $\phi_{j}$, to be thought of as the frequency of the wavefunction $\phi_{j}$. The automorphic kernel $K_{\nu}(x, y) = \sum_{\gamma \in \Gamma} k_{\nu}(x^{-1} \gamma y)$ acts as an integral operator on $L^{2}(\Gamma \backslash \mathfrak{h})$, and the $K_{\nu}$-eigenvalue of $\phi_{j}$ is $\widehat{k}_{\nu}(i \nu_{j})$.

\section{Two trace formulas}

\label{twotraceformulas}

As in \S \ref{notationsection}, let $\Gamma \subset G$ be the lattice coming from an Eichler order $R$ of square-free level in a quaternion division algebra $B$, $F \subset B$ a real quadratic field, $L = L_F \subset \mathfrak{h}$ the corresponding geodesic and $\ell$ the corresponding closed geodesic. Denote $\mathscr{P}(\phi_{j}) = \mathscr{P}_{\ell}(\phi_{j})$. Let $g_{0} A g_{0}^{-1} = A_{L}$ be the identity component of $\operatorname{Stab}_{G}(L)$.

We start from the spectral expansion of the automorphization of $k_{\nu}$:
\[
\sum_{j} \widehat{k}_{\nu}(i \nu_{j}) \phi_{j}(x) \phi_{j}(y) = \sum_{\gamma \in \Gamma} k_{\nu}(x^{-1} \gamma y) \,,
\]
where the convergence is uniform for $x$ and $y$ in compact sets.
Let $n \geq 1$ be an integer coprime to the discriminant $\Delta_{R}$, and apply $T_{n}$ to the $x$-variable to obtain
\begin{align} \label{amplifiedpretrace}
\sum_{j} \widehat{k}_{\nu}(i \nu_{j}) \widehat{T}_{n}(\phi_{j}) \phi_{j}(x) \phi_{j}(y) = \sum_{\eta \in R(n) / \pm 1} k_{\nu}(x^{-1} \eta y) \,.
\end{align}
Setting $x = y$ and integrating over $\Gamma \backslash \mathfrak{h}$ we obtain the amplified standard pre-trace formula:
\begin{align} \label{standardpretrace}
\sum_{j} \widehat{k}_{\nu}(i \nu_{j}) \widehat{T}_{n}(\phi_{j}) = \int_{\Gamma \backslash \mathfrak{h}} \sum_{\eta \in R(n) / \pm 1} k_{\nu}(x^{-1} \eta x) d x
\end{align}
On the other hand, to make periods appear, we may integrate \eqref{amplifiedpretrace} over $\ell \times \ell$ to get the amplified relative pre-trace formula:
\begin{align} \label{relativetrace}
\sum_{j} \widehat{k}_{\nu}(i \nu_{j}) \widehat{T}_{n}(\phi_{j}) \mathscr{P}(\phi_{j})^{2} & = \int_{\ell \times \ell} \sum_{\eta \in R(n) / \pm 1} k_{\nu}(x^{-1} \eta y) dx dy
\end{align}
We now identify a main term and an error term in both pre-trace formulas.

\subsection{Standard pre-trace formula}

Rewrite \eqref{standardpretrace} as
\begin{align} \label{mainpluserrorstandard}
\begin{split}
\sum_{j} \widehat{k}_{\nu}(i \nu_{j}) \widehat{T}_{n}(\phi_{j})
& = |(R(n) \cap \mathbb{Q}) / \pm 1| \cdot \operatorname{Vol}(\Gamma \backslash \mathfrak{h}) \cdot k_{\nu}(e) \\
& + \sum_{ \eta \in (R(n) - \mathbb{Q}) / \pm 1} \int_{\mathscr{F}} k_{\nu}(x^{-1} \eta x) d x
\end{split}
\end{align}
where $\mathscr{F} \subset \mathfrak{h}$ is a fundamental domain for the action of $\Gamma$. We want to bound the sum in the RHS. We follow \cite[\S 2]{iwaniec1995}. The argument changes slightly because our convolution operator $k_{\nu}$ is a different one: (1) The support of $k_{\nu}$ does not decrease with $\nu$, which makes that the contribution of hyperbolic elements cannot be ignored. (2) $k_{\nu}$ is not necessarily nonnegative, so we cannot bound the contribution of a single $\eta$ by the contribution of its conjugacy class.

\begin{remark} \label{remarkdifferentconvolutionoperator}
The reason why we (have to) make this choice of convolution operator $k_{\nu}$, is the following. Our ultimate goal is to prove an inequality of the form
\begin{equation} \label{ultimategoal}
\sum_{j} \widehat{k}_{\nu}(i \nu_{j}) w_{j} \mathscr{P}(\phi_{j})^{2} \geq C_{\nu} \cdot \sum_{j} \widehat{k}_{\nu}(i \nu_{j}) w_{j}
\end{equation}
for suitable nonnegative weights $w_{j}$, which will imply that at least one eigenfunction $\phi_{j_{0}}$ satisfies the bound $\mathscr{P}(\phi_{j_{0}})^{2} \geq C_{\nu}$. Here, we expect $C_{\nu}$ to be just slightly larger than $\nu^{-1}$. In particular, it decreases with $\nu$ (in contrast to the situation in \cite{iwaniec1995, milicevic2010}). We want to bound $C_{\nu}$ from below by $\nu_{j_{0}}^{-1}$. If $\nu_{j_{0}}$ is small compared to $\nu$ (less than $\nu^{0.99}$, say), then this is not possible. To ensure that $\nu_{j_{0}} \gg \nu$, we want to truncate the sum in the LHS of \eqref{ultimategoal} and keep only the terms with $\nu_{j_{0}} \asymp \nu$. This is why we have chosen $k_{\nu}$ with the property that it has rapid decay away from $\nu$. As a bonus, we obtain an eigenfunction whose spectral parameter $\nu_{j_{0}}$ lies in an interval of bounded length around $\nu$.
\end{remark}

For the contributing hyperbolic classes, we will need additional arithmetic information:

\begin{lemma} \label{boundondiscriminant} Let $n \geq 1$ and $\eta \in R(n)$. Suppose that there exists $\gamma \in R^{1}$, $x \in \mathscr{F}$ and $\nu \geq 0$ with $x^{-1} \gamma^{-1} \eta \gamma x \in \operatorname{supp} k_{\nu}$. Let $K = \mathbb{Q}(\eta) \subset B$ be the quadratic subfield generated by $\eta$, and let $\mathcal{O} = R \cap K$, which is an order in $K$. Then there exists a constant $C > 0$, independent of $n$ and $\eta$ such that $|\operatorname{Tr} \bar{\rho}(\eta)| \leq C$ and the discriminant $D_{\mathcal{O}}$ of $\mathcal{O}$ satisfies $|D_{\mathcal{O}}| \leq C n \left\vert |\operatorname{Tr} \bar{\rho}(\eta)| - 2 \right\vert$.
\end{lemma}

\begin{proof} Because $K \cap R \cong \gamma K \gamma^{-1} \cap R$, these orders have the same discriminant, and because $\operatorname{Tr}(\gamma^{-1} \eta \gamma) = \operatorname{Tr}(\eta)$ we may assume $\gamma = 1$. Then $\overline{\rho}(\eta)$ belongs to the bounded set $\mathscr{F} (\bigcup_{\nu} \operatorname{supp} k_{\nu}) \mathscr{F}^{-1}$. Thus $\frac{1}{\sqrt{n}} \rho(\eta)$ belongs to a bounded set independent of $n$ and $\nu$. By the remarks in \S \ref{distancesequivalent} there exists $C' > 0$ with $\Vert \frac{1}{\sqrt{n}} \eta \Vert \leq C'$, where $\Vert \cdot \Vert$ is the norm on $B \otimes_{\mathbb{Q}} \mathbb{R}$ we fixed in \S \ref{quaternionalgebranotation}.
Let $\alpha \in \mathcal{O}$ be such that $\alpha^{2} = D_{\mathcal{O}}$. Then $\mathcal{O} \subset \frac{1}{2} \mathbb{Z} + \frac{1}{2} \alpha \mathbb{Z}$. Write $\eta = a + b \alpha$ with $a, b \in \frac{1}{2} \mathbb{Z}$. Then
\[ |\operatorname{Tr}(\eta)| = \Vert 2 a \Vert = \Vert \eta + \bar{\eta} \Vert \leq 2 C' \sqrt{n} \]
so that $|\operatorname{Tr}(\bar{\rho}(\eta))| \leq 2 C'$. Because $\eta \notin \mathbb{Q}$ we have $b \neq 0$, so $|b|\geq 1 / 2$. From $n = N(\eta) = a^{2} - D_{\mathcal{O}} b^{2}$ we have
\[ |D_{\mathcal{O}}| = \frac{|a^{2} - n|}{b^{2}} \leq 4 |a^{2} - n| = n (|\operatorname{Tr} \bar{\rho}(\eta)| + 2) \left\vert|\operatorname{Tr} \bar{\rho}(\eta)| - 2 \right\vert \]
and we can take $C = 2 C' + 2$.
\end{proof}

Let $s$ denote the characteristic function of the squares, so that $|(R(n) \cap \mathbb{Q}) / \pm 1| = s(n)$.

\begin{proposition} \label{standardpretraceasymptotic} There exists an absolute constant $C > 0$ such that
\[ \sum_{j} \widehat{k}_{\nu}(i \nu_{j}) \widehat{T}_{n}(\phi_{j}) = s(n) \operatorname{Vol}(\Gamma \backslash \mathfrak{h}) \cdot k_{\nu}(e) + O \left( n^{3 / 2} e^{C \log n/ \log \log (n + 1)} \right) \,, \]
uniformly in $n \geq 1, \nu \geq 0$.
\end{proposition}

\begin{proof}
We may rewrite the sum in the RHS of \eqref{mainpluserrorstandard} as
\begin{equation} \label{standardpretracesplitintoconjugacy}
\frac{1}{2} \sum_{\mathcal{C}} \sum_{\eta \in \mathcal{C}} \int_{\mathscr{F}} k_{\nu}(x^{-1} \eta x) d x \,,
\end{equation}
where the sum runs over all $R^{1}$-conjugacy classes $\mathcal{C} \subset R(n) - \mathbb{Q}$. Fix $\mathcal{C}$ and take $\eta \in \mathcal{C}$. Because $B$ is a division algebra, $\overline{\rho}(\eta)$ is not parabolic.

When $\overline{\rho}(\eta)$ is hyperbolic, call $P > 1$ the square of its largest eigenvalue. Because $\eta$ generates a real quadratic quadratic subfield $K \subset B$, the centralizer $Z_{\Gamma}(\overline{\rho}(\eta))$ is an infinite cyclic group. Let $\eta_{0} \in K \cap R^{1}$ be totally positive and such that $\overline{\rho}(\eta_{0})$ generates this centralizer, and let $P_{0} > 1$ the square of the largest eigenvalue of $\overline{\rho}(\eta_{0})$. By the computation in \cite[\S 10.5]{iwaniec2002} and our choice of $k_{\nu}$,
\begin{align*}
\sum_{\eta \in \mathcal{C}} \int_{\mathscr{F}} k_{\nu}(x^{-1} \eta x) d x
& = \frac{\log P_{0}}{P^{1 / 2} - P^{-1 / 2}} \int_{\mathbb{R}} \widehat{k}_{\nu}(ir) e^{ir \log P} (2 \pi)^{-1} dr \\
& \ll \frac{\log P_{0}}{P^{1 / 2} - 1} \,.
\end{align*}
We have
\[ (P^{1 / 2} - 1)^{2} \geq \frac{(P^{1 / 2} - 1)^{2}}{P^{1 / 2}} = P^{1 / 2} + P^{-1 / 2} - 2 = |\operatorname{Tr}(\overline{\rho}(\eta)) | - 2 \,. \]
In order to bound $P_{0}$, let $\sigma : K \to \mathbb{R}$ be an embedding and observe that $P_{0}^{1 / 2} + P_{0}^{-1 / 2} = \operatorname{Tr}_{B / \mathbb{Q}}(\eta_{0}) = \sigma(\eta_{0}) + \sigma(\eta_{0})^{-1}$. Thus $P_{0} = \sigma(\eta_{0})^{\pm 2}$, so that $\log P_{0} = 2 \vert \log \sigma(\eta_{0}) \vert$. When $R_{\mathcal{O}}$ is the regulator of the order $\mathcal{O} = R \cap K$, we have $\vert \log \sigma(\eta_{0}) \vert \in \{ R_{\mathcal{O}}, 2 R_{\mathcal{O}} \}$ according to whether the fundamental unit of $\mathcal{O}$ has norm $\pm 1$. From the proof of Dirichlet's unit theorem, one has that $R_{\mathcal{O}}$ is bounded up to a constant by a power of the discriminant $D_{\mathcal{O}}$. Using the class number formula this can be improved to $D_{\mathcal{O}}^{1 / 2} \log D_{\mathcal{O}} \log \log D_{\mathcal{O}}$. (For non-maximal orders, this was done in \cite{sands1991}.) We may now assume that the sum over $\eta \in \mathcal{C}$ is nonzero, in which case $D_{\mathcal{O}} \ll n (|\operatorname{Tr} \bar{\rho}(\eta)| - 2) \ll n$ by Lemma~\ref{boundondiscriminant}. Combining this with the bound for $P^{1 / 2} - 1$, we obtain that for $\mathcal{C}$ hyperbolic,
\begin{equation} \label{hyperbolicboundoneclass}
\sum_{\eta \in \mathcal{C}} \int_{\mathscr{F}} k_{\nu}(x^{-1} \eta x) d x \ll \frac{n^{1 / 2} (|\operatorname{Tr} \bar{\rho}(\eta)| - 2)^{1 / 2} \log^{2} n}{(|\operatorname{Tr} \bar{\rho}(\eta)| - 2)^{1 / 2}} = n^{1 / 2} \log^{2} n \,.
\end{equation}
When $\overline{\rho}(\eta)$ is elliptic, let $0 < \theta \leq \pi / 2$ be such that $\pm e^{i \theta}$ is an eigenvalue. The centralizer $Z_{\Gamma}(\overline{\rho}(\eta))$ is a finite cyclic group. By the computation in \cite[\S 10.6]{iwaniec2002},
\begin{align*}
\sum_{\eta \in \mathcal{C}} \int_{\mathscr{F}} k_{\nu}(x^{-1} \eta x) d x
& = \frac{|Z_{\Gamma}(\overline{\rho}(\eta))|^{-1}}{ \sin \theta} \int_{\mathbb{R}} \widehat{k}_{\nu}(i r) \frac{\cosh((\pi - 2 \theta) r)}{\cosh(\pi r)} \frac{\pi}{2} dr \\
& \ll \frac{1}{\sin \theta} \, .
\end{align*}
Because $\bar{\rho}(\eta)$ is elliptic, $|\operatorname{Tr}(\eta)|^{2} < 4 n$, so $|\operatorname{Tr}(\eta)|^{2} \leq 4 n - 1$. We have
\[
2 - |\operatorname{Tr}(\bar{\rho}(\eta))| \geq 2 - \sqrt{4 - \frac{1}{n}} = \frac{1}{n \left( 2 + \sqrt{4 - \frac{1}{n}} \right)} \geq \frac{1}{4 n}
\]
so that
\[ 2 \sin^{2} \theta = 2 - 2 \cos^{2} \theta \geq 2 - 2 \cos \theta = 2 - |\operatorname{Tr}(\bar{\rho}(\eta))| \geq (4 n)^{-1} \,. \]
 We obtain that, for $\mathcal{C}$ elliptic,
\begin{equation} \label{ellipticboundoneclass}
\sum_{\eta \in \mathcal{C}} \int_{\mathscr{F}} k_{\nu}(x^{-1} \eta x) d x \ll n^{1 / 2} \,.
\end{equation}
We now count the number of contributing conjugacy classes. For the term in \eqref{standardpretracesplitintoconjugacy} corresponding to $\eta$ to be nonzero, there must exist $x$ in the compact set $\mathscr{F}$ with $x^{-1} \overline{\rho}(\eta) x \in \operatorname{supp} k_{\nu}$. We may thus restrict the sum \eqref{standardpretracesplitintoconjugacy} to $\mathcal{C}$ which contain an element $\eta$ with $\Vert \eta \Vert \leq C' \sqrt{n}$. The number of such conjugacy classes is bounded by the cardinality of the set
\[ M = \{ \eta \in R(n) : \Vert \eta \Vert \leq C' \sqrt{n} \} \ . \]
This is a counting problem that has arisen many times in the context of sup norms of Maass cusp forms. Letting $(\eta_{0}, \eta_{1}, \eta_{2}, \eta_{3})$ be a fixed $\mathbb{Z}$-basis of $R$ and using that $\Vert \cdot \Vert$ is equivalent to the $\sup$ norm w.r.t.\ this basis, we see that $\vert M \vert \ll (\sqrt{n})^{4}$. A more careful treatment gives $\vert M \vert \ll n \exp(C \log n / \log \log (n + 1))$ for some $C > 0$; see for example \cite[\S 5]{milicevic2010} or Remark~\ref{remarkcountallheckereturn} below. Multiplying this bound for $|M|$ by the larger one of the bounds \eqref{hyperbolicboundoneclass} and \eqref{ellipticboundoneclass}, the claim follows.
\end{proof}

\subsection{Relative pre-trace formula} 

\begin{proposition} \label{mainpluserrorworkable} There exists a nonnegative $b \in C_{c}^{\infty}(\mathbb{R})$ depending only on $\Gamma$ and $L$, such that for every $n \in \mathbb{N}_{> 0}$,
\begin{align*}
\sum_{j} \widehat{k}_{\nu}(i \nu_{j}) \widehat{T}_{n}(\phi_{j}) \mathscr{P}(\phi_{j})^{2} & = |N_{R(n)}(F) / R_{F}^{1}| \cdot \operatorname{Vol}(\Gamma_{L} \backslash L) \cdot \int_{\mathbb{R}} k_{\nu}(a(t)) dt \\
& + \sum_{\eta \in (R(n) - N_{R(n)}(F)) / \pm 1} I(\nu, g_{0}^{-1} \eta g_{0}) \,,
\end{align*}
where for $g \in G$ we define
\begin{equation} \label{definitionorbitalintegral}
I(\nu, g) = \int_{\mathbb{R} \times \mathbb{R}} b(s) b(t) k_{\nu}(a(-s) g a(t)) ds dt \,.
\end{equation}
\end{proposition}

\begin{proof} We start from the relative pre-trace formula \eqref{relativetrace}.
We have that $L$ has unit length parametrization by $g_{0} a(t) i$ where $a(t) = \exp \begin{psmallmatrix}
t / 2 & 0 \\
0 & -t / 2
\end{psmallmatrix}$. Let $\mathscr{F}_{L} \subset L$ be a fixed fundamental domain for the action of $\Gamma_{L}$. Write
\begin{align*}
\sum_{\eta \in R(n) / \pm 1} \int_{\mathscr{F}_{L} \times \mathscr{F}_{L}} k_{\nu}(x^{-1} \eta y) dx dy
& = \sum_{\eta \in N_{R(n)}(F) / \pm 1} \int_{\mathscr{F}_{L} \times \mathscr{F}_{L}} k_{\nu}(x^{-1} \eta y) dx dy \nonumber \\
+ & \sum_{\eta \in (R(n) - N_{R(n)}(F)) / \pm 1} \int_{\mathscr{F}_{L} \times \mathscr{F}_{L}} k_{\nu}(x^{-1} \eta y) dx dy \,.
\end{align*}
The first term equals, by unfolding the sum in the second variable and then making the change of variables $y \leftarrow \eta^{-1} y$,
\begin{align*}
\sum_{\eta \in N_{R(n)}(F) / R_{F}^{1} } & \int_{\mathscr{F}_{L} \times L} k_{\nu}(x^{-1} \eta y) dx dy \\
& = \sum_{\eta \in N_{R(n)}(F) / R_{F}^{1} } \int_{\mathscr{F}_{L} \times L} k_{\nu}(x^{-1} y) dx dy \\
& = |N_{R(n)}(F) / R_{F}^{1}| \int_{[0, \operatorname{Vol}(\mathscr{F}_{L})] \times \mathbb{R}} k_{\nu}(a(-s + t)) ds dt \\
& = |N_{R(n)}(F) / R_{F}^{1}| \cdot \operatorname{Vol}(\Gamma_{L} \backslash L) \cdot \int_{\mathbb{R}} k_{\nu}(a(t)) dt
\end{align*}
For the second term, we unfold the sum as follows:
\begin{align*}
& \sum_{\delta \in \Gamma_{L} \backslash \overline{\rho}(R(n) - N_{R(n)}(F)) / \Gamma_{L}} \sum_{\gamma \in \Gamma_{L} \delta \Gamma_{L}} \int_{\mathscr{F}_{L} \times \mathscr{F}_{L}} k_{\nu}(x^{-1} \gamma y ) dx dy \\
& = \sum_{\delta \in \Gamma_{L} \backslash \overline{\rho}(R(n) - N_{R(n)}(F)) / \Gamma_{L}} \int_{L \times L} k_{\nu}(x^{-1} \delta y ) dx dy
\end{align*}
Here, we used that for $\delta$ as above, the integral over $L \times L$ converges absolutely thanks to the compact support of $k_{\nu}$, and that the map $\Gamma_{L} \times \Gamma_{L} \to G : (\gamma_{1}, \gamma_{2}) \mapsto \gamma_{1} \delta \gamma_{2}$ is injective. Indeed, the contrary would imply that $\Gamma_{L} \cap \delta \Gamma_{L} \delta^{-1} \neq \{ 1 \}$, so in particular $A_{L} \cap \delta A_{L} \delta^{-1} \neq \{ 1 \}$. This would imply $A_{L} = \delta A_{L} \delta^{-1}$ so that $\delta \in N_{G}(A_{L})$. If $\delta = \overline{\rho}( \eta)$ with $\eta \in R(n)$, this means that $\eta \in N_{R^{+}}(F)$, a contradiction.

Now we introduce a smooth cutoff function in each of the integrals. Let $b_0 \in C_{0}^{\infty}(L)$ be $\Gamma_{L}$-invariant and such that $\sum_{\gamma \in \Gamma_{L}} b_0(\gamma z) = 1$ for all $z \in L$. Then the sum equals
\begin{align*}
& \sum_{\delta \in \Gamma_{L} \backslash \overline{\rho}(R(n) - N_{R(n)}(F)) / \Gamma_{L}} \int_{L \times L} \sum_{\gamma_{1}, \gamma_{2} \in \Gamma_{L}} b_0(\gamma_{1} x) b_0(\gamma_{2} y) k_{\nu}(x^{-1} \delta y ) dx dy \\
& \sum_{\delta \in \Gamma_{L} \backslash \overline{\rho}(R(n) - N_{R(n)}(F)) / \Gamma_{L}} \int_{L \times L} \sum_{\gamma_{1}, \gamma_{2} \in \Gamma_{L}} b_0(x) b_0(y) k_{\nu}(x^{-1} \gamma_{1}^{-1} \delta \gamma_{2} y ) dx dy \\
& = \sum_{\gamma \in (R(n) - N_{R(n)}(F)) / \pm 1} \int_{L \times L} b_0(x) b_0(y) k_{\nu}(x^{-1} \gamma y ) dx dy \,,
\end{align*}
where we made the change of variables $(x, y) \leftarrow (\gamma_{1} x, \gamma_{2} y)$ and merged the sum over $\delta$ with the sum over $\gamma_{1}, \gamma_{2}$. We obtain the statement of the proposition, with $b(t) := b_0(g_{0} a(t))$.
\end{proof}

For the integral appearing in the main term in Proposition~\ref{mainpluserrorworkable}, we shall prove in \S \ref{maintermanalysis}:

\begin{proposition} \label{integralmainterm} We have
\[ \int_{\mathbb{R}} k_{\nu}(a(t)) dt \asymp 1 \]
as $\nu \to \infty$.
\end{proposition}

Because the mass of $k_\nu$ is concentrated around $K$, we expect the integral $I(\nu, g)$ to be large when $a(-s) ga(t) \in K \backslash G / K$ is close to $K$ on a set of large measure in $\operatorname{supp} b \times \operatorname{supp} b$. That is, when the geodesic segments $\{ a(s) i : s \in \operatorname{supp} b \}$ and $\{ g a(t) i : t \in \operatorname{supp} b \}$ are close to each other. We see that the size of $I(\nu, g)$ should be related to the distance of $g$ to $N_{G}(A)$. This is quantified as follows:

\begin{proposition} \label{boundorbitalintegral} Define $I(\nu, g)$ by \eqref{definitionorbitalintegral}.
\begin{enumerate}
\item There exists $C' > 0$ independent of $\nu \geq 0$ such that $I(\nu, g) = 0$ unless $d(g, e) \leq C'$.
\item For $g \in G$, we have
\[ |I(\nu, g)| \ll (1 + \nu \cdot d(g, N_{G}(A)))^{-1 / 2} \]
uniformly in $g$ and $\nu \geq 0$.
\end{enumerate}
\end{proposition}

\begin{proof}
We prove the first assertion and postpone the proof of the second assertion until \S \ref{orbitalintegrals}. By construction, $k_{\nu}$ is supported on points at bounded distance from $e$. If $I(\nu, g)$ is nonzero, there exist $t, s \in \operatorname{supp} b$ with $k_{\nu}(a(-t) g a(s)) \neq 0$. Then $g$ belongs to the set $a(\operatorname{supp} b) (\operatorname{supp} k_{\nu}) a(- \operatorname{supp} b)$, which is bounded uniformly in $\nu$.
\end{proof}

\section{Counting Hecke returns}

\subsection{Counting stabilizers}

\label{countingstabilizers}

Let $F \subset B$ as before be the real quadratic number field corresponding to the geodesic $L$.
We want to understand the factor $|N_{R(n)}(F) / R_{F}^{1}|$, which appears in the main term in the relative pre-trace formula. Denote by $\mathfrak{f}_{R_{F}}$ the conductor of the order $R_{F}$, by $P_{R_{F}}$ the set of principal ideals of the maximal order $\mathcal{O}_{F}$ which are generated by an element of $R_{F}$, and by $P_{R_{F}}(n)$ the set of such ideals of norm $n$.

\begin{lemma} \label{cardinalitystabilizertau}When $n \geq 1$ is coprime to $\mathfrak{f}_{R_{F}}$, we have
\[ |N_{R(n)}(F) / R_{F}^{1}| \geq |P_{R_{F}}(n)| \]
\end{lemma}

\begin{proof}We may bound the LHS from below by $|R_{F}(n) / R_{F}^{1}|$. We show that for $(n, \mathfrak{f}_{R_{F}}) = 1$, this cardinality equals $|P_{R_{F}}(n)|$.
An element $\eta \in R_{F}(n)$ determines a principal ideal $\eta R_{F} \subset R_{F}$, which determines a principal ideal $\eta \mathcal{O}_{F} \in P_{R_{F}}(n)$. The composition of the two maps obtained in this way, is by definition surjective onto $P_{R_{F}}(n)$. The first map has fibers which are full orbits under multiplication by $R_{F}^{1}$. The second is injective, provided that $n$ is coprime to the conductor $\mathfrak{f}_{R_{F}}$. \end{proof}

\subsection{Bounding approximate stabilizers}

\label{countingalmostheckereturns}

 In order to control the sum appearing in the error term in Proposition~\ref{mainpluserrorworkable}, we will need an upper bound for the number of elements $\eta \in R(n)$ that are close to stabilizing $L$ without actually stabilizing it. That is, for the cardinality of the sets
\[ M(n, \delta) = \left\{ \eta \in R(n) : d(\bar{\rho}(\eta), e) \leq C', \, 0 < d(\bar{\rho}(\eta), N_{G}(A_{L})) \leq \delta \right\} \,, \]
where $C' > 0$ is an arbitrary constant which we fix throughout this section. 
This counting problem is similar to, but slightly different from the ones considered in \cite{iwaniec1995, marshall2016} in the context of upper bounds for sup norms resp.\ geodesic periods of Maass cusp forms: In the definition of $M(n, \delta)$, we are excluding the $\eta \in R(n)$ that stabilize $L$, so that the upper bound for $|M(n, \delta)|$ we obtain is smaller than what a direct invocation of \cite[Lemma 3.3]{marshall2016} would imply. This is necessary, because including those $\eta$ in the definition of $M(n, \delta)$ would force any upper bound to be at least as large as $|N_{R(n)}(F) / R_{F}^{1}|$, while we want the latter quantity to dominate the error term. Our method for bounding $|M(n, \delta)|$ however, uses many ideas which originate in \cite{iwaniec1995}.

\begin{lemma} \label{approximatesolutionindefinite} There exists an absolute constant $C > 0$ such that when $D \in \mathbb{Z}_{> 0}$ is a non-square, $0 < \delta \leq B$ and $n \geq 1$,
\begin{align*}
& \# \left\{ (u, v) \in \mathbb{Z}^{2} : 0 < \left| u^{2} - D v^{2} - n \right| \leq \delta n, \, |u|, |v| \leq B n \right\} \\
& \qquad \ll \delta n e^{C \log n / \log \log (n + 1)}
\end{align*}
where the implicit constant depends on $D$ and $B$.
\end{lemma}

\begin{proof} We want to estimate
\begin{align*}
& \# \{ (u, v) \in \mathbb{Z}^{2} : 0 < |u^{2} - D v^{2} - n| \leq \delta n, \, |u|, |v| \leq B n \} \\
& = \sum_{0 < |m - n| \leq \delta n} \# \{ u, v : u^{2} - D v^{2} = m, \, |u|, |v| \leq B n \}
\end{align*}
Let $K = \mathbb{Q}(\sqrt{D})$. Fix $m$ as in the sum. Any pair $(u, v)$ as above determines an element $z = u + v \sqrt{D} \in K$ of norm $m$, and if $|\cdot|_{1}$, $|\cdot|_{2}$ denote the two Archimedian absolute values of $K$, we have $|z|_{1}, |z|_{2} \leq B n (1 + \sqrt{D}))$. We obtain the upper bound
\begin{align*}
& \# \{ (u, v) \in \mathbb{Z}^{2} : u^{2} - D v^{2} = m , \, |u|, |v| \leq B n \} \\
& \leq \# \{ z \in \mathcal{O}_{K} : N(z) = m, |z|_{1}, |z|_{2} \leq B n (1 + \sqrt{D}) \} \,.
\end{align*}
The latter set maps to the set of integral ideals in $\mathcal{O}_{K}$ with norm $m$. There are at most $\tau(m)$ such ideals. We may bound the divisor function $\tau(m)$ by its maximal order $\exp(C \log m / \log \log (m + 1)) \ll_\delta \exp(C \log n / \log \log (n + 1))$ for some $C > 0$. The fibers of said map are orbits under multiplication by units of norm $1$. Let $\varepsilon$ be the fundamental unit of $K$. The bound on $|z|_{1}$ and $|z|_{2}$ implies that the fibers of that map are of size at most
\begin{align*}
\ll \frac{\log (B n (1 + \sqrt{D}))}{\log|\varepsilon|}
& \ll 1 + \log n
\end{align*}
We obtain the upper bound
\begin{align*}
& \sum_{0 < |m - n| \leq \delta n} \# \{ u, v : u^{2} - D v^{2} = m, \, |u|, |v| \leq B n \} \\
& \ll \delta n \cdot \exp(C \log n / \log \log (n + 1)) \cdot (1 + \log n) \,,
\end{align*}
as desired. Here we used that the number of terms in the sum is $\ll \delta n$, and not just $\ll \delta n + 1$, since we omit the term for $m = n$.
\end{proof}

\begin{lemma} \label{countalmostheckereturnsA} There exists an absolute constant $C > 0$ such that
\[ \# M(n, \delta) \ll \delta n e^{C \log n / \log \log (n + 1)} \,. \]
\end{lemma}

\begin{proof}Let $D > 0$ be the discriminant of $F$ and take $\alpha \in F$ with $\alpha^{2} = D$. For $t, u \in \mathbb{R}$ we have $N_{M_{2}(\mathbb{R}) / \mathbb{R}}(t + u \rho(\alpha)) = t^{2} - Du^{2}$, so that
\begin{align} \label{parametrizationstabilizergeodesic}
A_{L} = \bar{\rho}((F \otimes_{\mathbb{Q}} \mathbb{R})^{1}) = \left\{ \overline{t + u \rho(\alpha)} : t^{2} - D u^{2} = 1 \right\}
\end{align}
Let $\omega \in R^{+}$ be the Skolem--Noether element from \S \ref{notationclosedgeodesics}. Multiplying $\omega$, if necessary, by an element of negative norm in $F$ and by a suitable integer, we may assume that $\omega \in R^{+}$. Let $E := N_{B / \mathbb{Q}}(\omega) = - \omega^{2} \in \mathbb{Z}_{> 0}$. Let $S \subset B$ be the order $\mathbb{Z} + \mathbb{Z} \alpha + \mathbb{Z} \omega + \mathbb{Z} \alpha \omega$. Let $f \in \mathbb{N}_{> 0}$ be such that $f \cdot R \subset S$. Now let $\eta \in M(n, \delta)$. We can write $\eta = x_{0} + x_{1} \alpha + x_{2} \omega + x_{3} \alpha \omega$ with $x_{0}, x_{1}, x_{2}, x_{3} \in \frac{1}{f} \mathbb{Z}$, and we have $N_{B / \mathbb{Q}}(\eta) = x_{0}^{2} - D x_{1}^{2} + E x_{2}^{2} - DE x_{3}^{2} = n$.

Suppose first that $d(\bar{\rho}(\eta), A_{L}) \leq \delta$. Let $C'$ be the constant from the beginning of the section, with respect to which $M(n, \delta)$ is defined. Because $M(n, \delta) = M(n, C')$ for $\delta \geq C'$, it is no restriction to assume $\delta \leq C'$. By the remarks in \S \ref{distancesequivalent}, from $d(\bar{\rho}(\eta), A_{L}) \leq \delta$ and $d(\bar{\rho}(\eta), e) \leq C'$ it follows that there exists $a \in \operatorname{SL}_{2}(\mathbb{R})$ with $\overline{a} \in A_{L}$ and $\Vert \frac{1}{\sqrt{n}} \rho(\eta) - a \Vert \ll \delta$. Here, the norm is any fixed norm on $M_{2}(\mathbb{R})$. Writing $a = t + u \rho(\alpha)$ as in \eqref{parametrizationstabilizergeodesic}, it follows that in $B \otimes_{\mathbb{Q}} \mathbb{R}$,
\[ \frac{1}{\sqrt{n}}(\eta \otimes 1) = 1 \otimes t + \alpha \otimes u + O(\delta) \]
and comparing coordinates in the basis $(1, \alpha, \omega, \alpha \omega)$, we obtain
\begin{align}
\begin{split}
\frac{x_{0}}{\sqrt{n}} & = t + O(\delta) \\
\frac{x_{1}}{\sqrt{n}} & = u + O(\delta) \\
\frac{x_{2}}{\sqrt{n}} & = O(\delta) \\
\frac{x_{3}}{\sqrt{n}} & = O(\delta) \,,
\end{split}
\end{align}
so that
\[
1 = t^{2} - Du^{2} = \frac{x_{0}^{2}}{n} - D \frac{x_{1}^{2}}{n} + O(\delta) \,.
\]
That is, $|x_{0}^{2} - D x_{1}^{2} - n| \ll \delta n$. The assumption $d(\bar{\rho}(\eta), e) \leq C'$ implies that $x_{0}, x_{1}, x_{2}, x_{3} \ll \sqrt{n}$. Because $d(\bar{\rho}(\eta), N_{G}(A)) > 0$, we have $\eta \notin F$, so that $(x_{2}, x_{3}) \neq (0, 0)$, so that $x_{0}^{2} - D x_{1}^{2} = n - E (x_{2}^{2} - D x_{3}^{2}) \neq n$. We obtain that the integers $(fx_{0}, fx_{1})$ satisfy
\[ 0 < |(f x_{0})^{2} - D (f x_{1})^{2} - f^{2} n| \ll \delta n \]
and $|f x_{0}|, |f x_{1}| \ll n$.
By Lemma~\ref{approximatesolutionindefinite} and the assumption that $\delta \leq C'$, the number of possible values for $(x_{0}, x_{1})$ is at most $\ll \delta n e^{C \log n / \log \log (n + 1)}$ for some $C > 0$. If we fix $x_{0}$ and $x_{1}$, then $x_{2}$ and $x_{3}$ satisfy
\begin{equation} \label{equationxtwothree}
|(f x_{2})^{2} - D (f x_{3})^{2}| = f^{2}/E \cdot |n - N_{B / \mathbb{Q}}(x_{0} + x_{1} \alpha)| \ll n \end{equation}
and $|f x_{2}|, |f x_{3}| \ll n$.
By counting ideals in $\mathbb{Q}(\sqrt{D})$ as in the proof of Lemma~\ref{approximatesolutionindefinite}, we see that the number of $(x_{2}, x_{3})$ satisfying the equality in \eqref{equationxtwothree} can be bounded by $\ll \exp(C \log n / \log \log (n + 1))$. This proves that the number of $\eta \in M(n, \delta)$ with $d(\bar\rho(\eta), A_{L}) \leq \delta$ is at most $\ll \delta n e^{2 C \log n / \log \log (n + 1)}$.

Suppose now that $\eta \in M(n, \delta)$ is such that $d(\overline{\rho}(\eta), N_{G}(A_{L}) - A_{L}) \leq \delta$. We have $\overline{\rho}(\omega) \in N_{G}(A_{L}) - A_{L}$, so we obtain that $d(\overline{\rho}(\omega \eta), A_{L}) \ll \delta$, $d(\bar{\rho}(\omega \eta), e) \ll 1$ and $\omega \eta \in R(n \cdot E)$. By the first case (with a different constant $C'$ in the definition of $M(nE, \delta)$), the number of such $\eta$ is bounded by $\ll \delta nE e^{C \log (n E) / \log \log (n E + 1)}$ for some $C > 0$.

Adding the contributions from the two types of $\eta$, the claim follows.
\end{proof}

\begin{remark} \label{remarkcountallheckereturn}
\begin{enumerate}
\item It is clear from the proofs of Lemma's~\ref{approximatesolutionindefinite} and \ref{countalmostheckereturnsA} that when $M' (n, \delta)$ is the set obtained by removing the condition $0 < d(\bar{\rho}(\eta), N_{G}(A))$ in the definition of $M(n, \delta)$, we obtain the upper bound
\[ \# M' (n, \delta) \ll (\delta n + 1) e^{C \log n / \log \log (n + 1) } \]
for some $C > 0$. In particular,
\[ \# \left\{ \eta \in R(n) : d(\bar{\rho}(\eta), e) \leq C' \right\} = M' (n, C') \ll n e^{C \log n / \log \log (n + 1)} \,.\]
\item 
Note that, as opposed to \cite[Lemma 3.3]{marshall2016}, our bound for $M(n, \delta)$ is not uniform in the geodesic $L$, which explains why we are able to get a factor $\delta$ instead of only $\sqrt{\delta}$. But this improvement will not play a major role when we use this bound in the proof of Lemma~\ref{boundsumorbitalintegral} below. In fact, the larger part of our estimate of the error term in Proposition~\ref{mainpluserrorworkable} will come from the $\eta$ with $d(\bar{\rho}(\eta), N_{G}(A_{L})) \asymp 1$, for which the higher power of $\delta$ in Lemma~\ref{countalmostheckereturnsA} gives no improvement.
\end{enumerate}
\end{remark}

\section{Stationary phase}
\label{maintermanalysis}

\begin{proof}[Proof of Proposition~\ref{integralmainterm}] Let $c : \mathbb{R} \to [0, 1]$ be smooth, compactly supported and such that $c(t) = 1$ when $a(t) \in \bigcup_{\nu \geq 0} \operatorname{supp} k_{\nu}$. By \eqref{inverseHarishChandra} and Fubini,
\begin{align*}
\int_{\mathbb{R}} k_{\nu}(a(t)) dt & =
\int_{\mathbb{R}} k_{\nu}(a(t)) c(t) dt \\
& = \int_{\mathbb{R}} \int_{0}^{\infty} \widehat{k}_{\nu}(ir) \varphi_{ir}(a(t)) \beta(r) c(t) d r dt \\
& = \int_{0}^{\infty} \widehat{k}_{\nu}(ir) \beta(r) \int_{\mathbb{R}} c(t) \varphi_{ir}(a(t)) dt dr
\end{align*}
because the compact support of $c$ and the rapid decay of $\widehat{k}_{\nu}$ make the double integral absolutely convergent. Let
\begin{equation} \label{definitionmodelintegral}
\mathscr{L}(r) = \int_{\mathbb{R}} c(t) \varphi_{ir}(a(t)) dt \,,
\end{equation}
which is a real number.
By Proposition~\ref{maintermmodelintegral} below, there exist $r_{0}, C > 0$ such that $\beta(r) \mathscr{L}(r) \in [C / 2, 2 C]$ for $r \geq r_{0}$. Write
\begin{align*}
& \int_{0}^{\infty} \widehat{k}_{\nu}(ir) \beta(r) \mathscr{L}(r) dr \\
& = \int_{0}^{r_{0}} \widehat{k}_{\nu}(ir) \beta(r) \mathscr{L}(r) dr + \int_{r_{0}}^{\infty} \widehat{k}_{\nu}(ir) \beta(r) \mathscr{L}(r) dr \,.
\end{align*}
For $\nu > r_{0}$, the first term is bounded by $r_{0} (\nu - r_{0})^{-1} \cdot \sup_{r \in [0, r_{0}]} \beta(r) |\mathscr{L}(r)|$, and thus $o(1)$ as $\nu \to \infty$.
Because $\beta(r) \mathscr{L}(r) \in [C / 2, 2 C]$ for $r \geq r_{0}$, the second term is positive and asymptotically equivalent to $\int_{r_{0}}^{\infty} \widehat{k}_{\nu}(ir) dr$. This integral is bounded from above by
\begin{align*}
& \ll \int_{r_{0}}^{\infty} (1 + |\nu - r|)^{-2} dr \\
& \leq \int_{- \infty}^{\infty} (1 + |r|)^{-2} dr \\
& \ll 1
\end{align*}
and when $\nu > r_{0}$ it is bounded from below by
\[ \int_{\nu}^{\nu + 1} \widehat{k}_{\nu}(ir) dr \geq 1 \,. \]
Hence $\int_{\mathbb{R}} k_{\nu}(a(t)) dt \asymp \int_{r_{0}}^{\infty} \widehat{k}_{\nu}(ir) dr \asymp 1$.
\end{proof}

\begin{proposition} \label{maintermmodelintegral}
Define $\mathscr{L}(r)$ by \eqref{definitionmodelintegral}. There exists $C > 0$ such that
\begin{align*}
\mathscr{L}(r) = C r^{-1} + O(r^{-2})
\end{align*}
as $r \to \infty$.
\end{proposition}

\begin{proof}
We have \cite[(1.43)]{iwaniec2002}
\[ \varphi_{ir} = \int_{S^{1}} (\cosh t + x_{1} \sinh t)^{i r - \frac{1}{2}} d \mu(x) \]
where $S^{1} = \{ (x_{1}, x_{2}) \in \mathbb{R}^{2} : x_{1}^{2} + x_{2}^{2} = 1 \}$ and the integral is w.r.t.\ the uniform measure.
Hence
\begin{align*}
\mathscr{L}(r) = \int_{\mathbb{R}} \int_{S^{1}} c(t) (\cosh t + x_{1} \sinh t)^{i r - \frac{1}{2}} d \mu(x) dt \,.
\end{align*}
By the stationary phase theorem \cite[\S VIII.2]{stein1993},
the asymptotic behavior of this integral is prescribed by the local properties of the integrand at the critical points of the phase function
\[ \phi(x, t) = \log( \cosh t + x_{1} \sinh t) \,. \]

\begin{lemma} The critical points of $\phi$ are $((\pm 1, 0), 0)$, and they are nondegenerate. The Hessian in the $(x_{1}, t)$-coordinates at those points is of shape
\[
\mathbf{H}
= \begin{pmatrix}
0 & 1 \\
1 & *
\end{pmatrix} \,,
\]
which has signature $(p, q) = (1, 1)$.
\end{lemma}

\begin{proof} At the points where $x_{0} = 0$ we have $\partial \phi / \partial t = \frac{\sinh t + x_{1} \cosh t}{\cosh t + x_{1} \sinh t} \neq 0$, because $|\cosh t| > |\sinh t|$. Hence they cannot be critical points.
In the open set where $x_{0} \neq 0$, we have $\partial \phi / \partial x_{1} = \frac{\sinh t}{\cosh t + x_{1} \sinh t}$. This is zero only when $t = 0$. We have $\partial \phi / \partial t = \frac{\sinh t + x_{1} \cosh t}{\cosh t + x_{1} \sinh t}$, which does not vanish when $t = 0$, unless $x_{1} = 0$. Thus the only critical points are $((\pm 1, 0), 0)$.

We have $\phi(x, 0) = 0$, so that at the points where $t = 0$ one has $\partial^{2} \phi / \partial x_{1}^{2} = 0$. At those points we have $\partial \phi / \partial t = x_{1}$, so that $\partial^{2} \phi / \partial t \partial x_{1} = 1$. This proves that the Hessian at the critical points of $\phi$ is of shape $\mathbf{H}$.
\end{proof}

Adding the contributions of the two critical points, stationary phase implies that $\mathscr{L}(r) = C r^{-1} + O(r^{-2})$
with
\[
C = 2 \cdot c(0) \cdot 2 \pi \cdot |\det(\mathbf{H})|^{-1 / 2} \cdot e^{i \pi (p - q) / 4} = 4 \pi \,. \qedhere
\]
\end{proof}

\section{Orbital integrals}

\label{orbitalintegrals}

We now prove Proposition~\ref{boundorbitalintegral}. For the most part, this corresponds to the `frequency $0$' case in \cite[\S 7]{marshall2016}, whose method we follow. Throughout, let $C'$ be the constant from Proposition~\ref{boundorbitalintegral}. 
Expanding $k_{\nu}$ in terms of spherical functions (see \eqref{inverseHarishChandra}), we have for $\nu \geq 0$ and $g \in G$,
\[ I(\nu, g) = \int_{0}^{\infty} J(r, g) \widehat{k}_{\nu}(ir) \beta(r) dr \]
where
\[ J(r, g) = \int_{K} \int_{\mathbb{R} \times \mathbb{R}} b(s) b(t) e^{(1 / 2 + ir) H(ka(-s) g a(t))} ds dt \, dk \]
and $H$ is as in \S \ref{sectionconvolutionoperators}. 
We will use the stationary phase method to prove an upper bound for $J(r, g)$. Proposition~\ref{boundorbitalintegral} will then follow by integrating this bound against $\widehat{k}_{\nu}(ir) \beta(r)$.

For each $g \in G$, we have a map $\alpha_{g} : K \to K$ defined by requiring that
\[ k g \in NA \alpha_{g}(k) \,. \]
It is smooth in $(g, k)$ because the Iwasawa decomposition $G \overset{\sim}{\to} N \times A \times K$ is smooth, and we have $\alpha_{gh} = \alpha_{h} \circ \alpha_{g}$. For $k \in K$ and $y, z \in G$ we have that (see e.g.\ \cite[Lemma 6.2]{marshall2016})
\[ H(k y^{-1} z) = H(\alpha_{y^{-1}}(k) z) - H(\alpha_{y^{-1}}(k) y) \,. \]
This allows us to separate variables inside the exponential in the definition of $J(r, g)$:
\begin{align*}
J(r, g) & = \int_{K} \int_{\mathbb{R} \times \mathbb{R}} b(s) b(t) e^{(1 / 2 + ir) H(ka(-s) g a(t))} ds dt \, dk \\
& = \frac{1}{2 \pi} \int_{0}^{2 \pi} \int_{\mathbb{R} \times \mathbb{R}} b(s) b(t) e^{(1 / 2 + ir) H(k(\sigma) a(-s) g a(t))} ds dt \, d \sigma \\
& = \frac{1}{2 \pi} \int_{0}^{2 \pi} \int_{\mathbb{R} \times \mathbb{R}} b' (s, \theta) b(s) b(t) e^{(1 / 2 + ir) \left( H(k(\theta) g a(t)) - H(k(\theta) a(s)) \right)} ds dt \, d \theta \,,
\end{align*}
by substituting $k(\theta) = \alpha_{a(-s)}(k(\sigma))$. Here, $b'$ is a nonzero smooth function on $\mathbb{R} \times \mathbb{R} / 2 \pi \mathbb{Z}$, coming from the change of variables. This oscillating integral has phase function
\begin{equation} \label{definitionphi}
\phi(s, t, \theta, g) := H(k(\theta) g a(t)) - H(k(\theta) a(s)) \,.
\end{equation}
Define $c(s, t, \theta, g) = (2 \pi)^{-1} b' (s, \theta) b(s) b(t) \exp(\phi(s, t, \theta, g) / 2)$ so that
\[
J(r, g) = \int_{0}^{2 \pi} \int_{\mathbb{R} \times \mathbb{R}} c(s, t, \theta, g) e^{ir \phi(s, t, \theta, g)} ds dt \, d \theta \,.
\]

\subsection{Critical points of \texorpdfstring{$\phi$}{phi}}

We analyze the critical points of $\phi$ for fixed $\theta$ and $g$. Let $\mathcal{P} = \mathbb{R} / 2 \pi \mathbb{Z} \times G$ and let $\mathcal{S} \subset \mathcal{P}$ be the closed subset consisting of `singular' parameters $(\theta, g)$ for which one of the geodesics $k(\theta) A i$ and $k(\theta) g A i$ is vertical. For every $g$, there are at most four values of $\theta$ for which $(\theta, g) \in \mathcal{S}$. 

\begin{proposition}When $(\theta, g) \in \mathcal{S}$ is fixed, $\phi$ has no critical points. When $(\theta, g) \in \mathcal{P} - \mathcal{S}$, $\phi$ has a unique critical point with Hessian given in $(s, t)$-coordinates by
\[ \begin{pmatrix}
\frac{1}{2} & 0 \\
0 & - \frac{1}{2}
\end{pmatrix} \,.
\]
\end{proposition}

\begin{proof}When $g$ and $\theta$ are fixed, we have that $(s, t)$ is a critical point of $\phi$ iff
\[ \left. \frac{\partial H(k(\theta) a(s))}{\partial s} \right|_{s = \xi_{1}(\theta)} = \left. \frac{\partial H(k(\theta) g a(t))}{\partial t} \right|_{t = \xi_{2}(\theta, g)} = 0 \,. \]
That is, iff the geodesics $k(\theta) A i$ and $k(\theta) g A i$ in $\mathfrak{h}$ are not vertical, i.e., are half-circles, and their midpoints are $k(\theta) a(s) i$ resp.\ $k(\theta) g a(t) i$. (Recall that $H(g) = \log(\operatorname{Im}(gi))$.)
It is clear from the geometric interpretation that $\phi$ has no critical points when $(\theta, g) \in \mathcal{S}$, and has exactly one critical point when $(\theta, g) \in \mathcal{P} - \mathcal{S}$. Moreover, this critical point is then nondegenerate, because the above geodesics (which are half-circles) have nonzero (Euclidean) curvature at their midpoints. Finally, the shape of the Hessian is computed in the proof of \cite[Lemma 7.10]{marshall2016}.
\end{proof}

When $(\theta, g) \in \mathcal{P} - \mathcal{S}$, let $(\xi_{1}(\theta), \xi_{2}(\theta, g))$ be the unique critical point of $\phi$.
The following lemma says that this point diverges to $\infty$ as $(\theta, g)$ approaches $\mathcal{S}$:

\begin{lemma} \label{supportsubstitutedcutoffcompact} Let $R > 0$. Suppose $g \in G$ with $d(g, e) \leq C'$ and $\theta \in \mathbb{R} / 2 \pi \mathbb{Z}$ are such that $|\xi_{1}(\theta)|, |\xi_{2}(\theta, g)| \leq R$. Then there exists $\delta = \delta(R, C') > 0$ such that $d((\theta, g), \mathcal{S}) > \delta$.
In particular, the set
\[ \mathcal{P}_{0} = \{ (\theta, g) \in \mathcal{P} - \mathcal{S} : (\xi_{1}(\theta), \xi_{2}(\theta, g), \theta, g) \in \operatorname{supp} c , \, d(g, e) \leq C' \} \]
is at positive distance from $\mathcal{S}$, and is compact.
\end{lemma}

\begin{proof} For the sake of contradiction, suppose that $d((\theta, g), \mathcal{S})$ could be arbitrarily small. Then there exists a sequence $(\theta_{n}, g_{n})$ for which $(\theta_{n}, g_{n}, \xi_{1}(\theta_{n}, g_{n}), \xi_{2}(\theta_{n}, g_{n}))$ converges in $\mathbb{R} / 2 \pi \mathbb{Z} \times G \times \mathbb{R}^{2}$, with a limit in $\mathcal{S} \times \mathbb{R}^{2}$. Call $(\theta, g, x, y)$ its limit. By continuity of ${\partial H(k(\theta) a(s))} / {\partial s}$ and ${\partial H(k(\theta) g a(t))} / {\partial t}$, we would have
\[ \left. \frac{\partial H(k(\theta) a(s))}{\partial s} \right|_{s = x} = \left. \frac{\partial H(k(\theta) g a(t))}{\partial t} \right|_{t = y} = 0 \,, \]
contradicting that one of the geodesics $k(\theta) A i$, $k(\theta) g A i$ is vertical. It follows that $d(\mathcal{P}_{0}, \mathcal{S}) > 0$. Because $\mathcal{P}_{0}$ is bounded in $\mathcal{P}$ and closed in $\mathcal{P} - \mathcal{S}$, it is compact.
\end{proof}

For $(\theta, g) \in \mathcal{P} - \mathcal{S}$, define
\begin{equation} \label{definitionpsi} \psi(\theta, g) = \phi(\xi_{1}(\theta), \xi_{2}(\theta, g), \theta, g) \,. \end{equation}
Define also
\begin{align*}
c_{1}(\theta, g) = \begin{cases}
\pi c(\xi_{1}(\theta), \xi_{2}(\theta, g), \theta, g) & : (\theta, g) \in \mathcal{P} - \mathcal{S} \,, \\
0 & : (\theta, g) \in \mathcal{S} \,.
\end{cases}
\end{align*}
Lemma~\ref{supportsubstitutedcutoffcompact} implies that $c_{1}$ is smooth on $\mathcal{P}$. Let $C'$ be as in Proposition~\ref{boundorbitalintegral} and $\mathcal{P}_{0}$ as in Lemma~\ref{supportsubstitutedcutoffcompact}.

\begin{lemma} \label{reductiontothetaintegral} We have
\begin{equation} \label{Jonemorevariable}
J(r, g) = r^{-1} \int_{0}^{2 \pi} c_{1}(\theta, g) e^{ir \psi(\theta, g)} d \theta + O(r^{-2})
\end{equation}
as $r \to \infty$,
where the implicit constant remains bounded as long as $d(g, e) \leq C'$.
\end{lemma}

\begin{proof} For fixed $\theta, g$ we apply the method of stationary phase in the variables $s$ and $t$: Suppose $(\theta, g) \in \mathcal{P} - \mathcal{S}$. Then $\phi$ has a unique critical point with Hessian determinant $-1 / 4$. Stationary phase \cite[\S VIII.2]{stein1993} implies
\begin{align} \label{statphaseinst}
\int_{\mathbb{R} \times \mathbb{R}} c(s, t, \theta, g) e^{ir \phi(s, t, \theta, g)} ds dt = r^{-1} e^{ir \psi(\theta, g)} c_{1}(\theta, g) + O(r^{-2}) \,,
\end{align}
uniformly for $(\theta, g)$ in compact subsets of $\mathcal{P} - \mathcal{S}$. Suppose $(\theta, g) \in \mathcal{P} - \mathcal{P}_{0}$, so that $\phi$ has no critical point in the support of $c$. The absence of critical points implies
\[
\int_{\mathbb{R} \times \mathbb{R}} c(s, t, \theta, g) e^{ir \phi(s, t, \theta, g)} ds dt \ll_{N} r^{-N} \,,
\]
uniformly for $(\theta, g)$ in compact subsets of $\mathcal{P} - \mathcal{P}_{0}$. Because $c_{1}(\theta, g) = 0$ in this case, we see that \eqref{statphaseinst} still holds (after assigning any value to $\psi(\theta, g)$ when $(\theta, g) \in \mathcal{S}$).
We may now take compact subsets of $\mathcal{P} - \mathcal{S}$ and of $\mathcal{P} - \mathcal{P}_{0}$ such that the union of the two contains $\mathbb{R} / 2 \pi \mathbb{Z} \times \{ g \in G : d(g, e) \leq C' \}$. Integrating \eqref{statphaseinst} w.r.t.\ $\theta$ then yields the desired estimate.
\end{proof}

\subsection{Critical points of \texorpdfstring{$\psi$}{psi}}

In view of the expression \eqref{Jonemorevariable} for $J(r, g)$, we analyze the critical points of $\psi$, for fixed $g$. When $(\theta, g) \notin \mathcal{S}$, we have \cite[Proposition 7.2, Lemma 7.3]{marshall2016} that $\theta$ is a critical point of $\psi$ iff the geodesics $k(\theta) Ai$ and $k(\theta) gAi$, which by assumption are half-circles, are concentric. A critical point $\theta$ is degenerate iff these geodesics coincide, that is, iff $g \in N_{G}(A)$, in which case every $\theta$ for which $(\theta, g) \notin \mathcal{S}$ is a degenerate critical point.

\begin{remark}For $g \notin N_{G}(A)$, the locations of the critical points of $\psi( -, g)$ can be described geometrically: as stated above, $\theta$ is a critical point iff the half-circles $k(\theta) Ai$ and $k(\theta) gAi$ are concentric, that is iff they share a common perpendicular vertical geodesic. In particular, the geodesics $k(\theta) Ai$ and $k(\theta) gAi$ must not intersect in $\mathfrak{h} \cup \mathbb{P}^{1}(\mathbb{R})$, that is, the geodesic $gAi$ must have both endpoints in $\mathbb{R}_{> 0}$ or in $\mathbb{R}_{< 0}$. This condition on $g$ is also sufficient for the existence of a critical point: Under this assumption on $g$, there are exactly two critical points which can be described as follows. Because $Ai$ and $gAi$ do not intersect in $\mathfrak{h} \cup \mathbb{P}^{1}(\mathbb{R})$, these geodesics lie at positive distance from each other, and this distance is realized in $\mathfrak{h}$ by a pair of points on $Ai$ and $gAi$. Let $j$ be the geodesic carrying the geodesic segment joining those points. As a consequence of Gauss's lemma, $j$ is then perpendicular to both geodesics. Because the sum of the angles of a hyperbolic quadrilateral is less than $360^{\circ}$, there can be no other geodesic that is perpendicular to both. Thus the critical points are the two values of $\theta$ for which $k(\theta) j$ is vertical.\end{remark}

To bound $J(r, g)$, we distinguish two cases, depending on whether $g$ is close to $N_G(A)$ or at positive distance from it.
For $g$ away from $N_{G}(A)$, the absence of nondegenerate critical points of $\psi$ implies:

\begin{lemma} \label{corputawayfromnormalizer} When $d(g, N_{G}(A)) \geq \delta > 0$ and $d(g, e) \leq C'$, we have
\[ \int_{0}^{2 \pi} c_{1}(\theta, g) e^{ir \psi(\theta, g)} d \theta \ll_{\delta} (1 + r \cdot d(g, N_{G}(A)))^{-1 / 2} \]
uniformly in $g$ and $r \geq 0$.
\end{lemma}

\begin{proof}
Because $\psi$ has no degenerate critical points when $g \notin N_{G}(A)$, we have $\max(|\partial \psi / \partial \theta|, |\partial^{2} \psi / \partial \theta^{2}|) \gg 1$ uniformly for $g$ in compact subsets of $G - N_{G}(A)$ and for $\theta \in \operatorname{supp} c_{1}(-, g)$. The Van der Corput lemma \cite[\S VIII.1, Proposition 2, Corollary]{stein1993} implies that
\[ \int_{0}^{2 \pi} c_{1}(\theta, g) e^{ir \psi(\theta, g)} d \theta \ll r^{-1 / 2} \]
as $r \to \infty$, uniformly for $g$ in compact subsets of $G - N_{G}(A)$.
\end{proof}

We may now restrict our attention to $g$ that are close to $N_{G}(A)$. By \eqref{definitionphi} and \eqref{definitionpsi},
\[
\psi(\theta, g) := H(k(\theta) g a(\xi_2(\theta, g))) - H(k(\theta) a(\xi_1(\theta))) \,.
\]
From the characterization of $k(\theta) g a(\xi_{2}(\theta, g)) i$ as the midpoint of the geodesic $k(\theta) g Ai$, we see that $\psi$ is right $N_{G}(A)$-invariant in $g$. Thus in order to estimate the RHS of \eqref{Jonemorevariable}, we will first assume that $g$ lies in a small neighborhood of the identity, and then translate the estimate to a small neighborhood of $N_{G}(A)$.

By explicating the order of vanishing of $\partial^{2} \psi(\theta, g) / \partial \theta^{2}$ as $d(g, A) \to 0$, an application of the Van der Corput lemma shows that:

\begin{lemma} \label{firstlemmadistance} \cite[Corollary 7.6]{marshall2016} There is an open neighborhood $U$ of $e \in G$ such that for all $b \in C_{c}^{\infty}(\mathbb{R} / 2 \pi \mathbb{Z} - \{\theta : (\theta, g) \in \mathcal{S} \} )$ and $g \in U$ we have
\[ \int_{0}^{2 \pi} b(\theta) e^{i r \psi(\theta, g)} d \theta \ll (1 + r \cdot d(g, A))^{-1 / 2} \,, \]
where the implicit constant remains bounded when $\operatorname{supp}(b)$ stays at positive distance from the set $\{ \theta : (\theta, g) \in \mathcal{S} \}$ and the derivatives of $b$ up to a certain order remain bounded.
\end{lemma}

\begin{corollary} \label{corputatnormalizer} There exists an open neighborhood $V$ of $N_{G}(A)$ in $G$ such that when $g \in V$ and $d(g, e) \leq C'$,
\[ \int_{0}^{2 \pi} c_{1}(\theta, g) e^{ir \psi(\theta, g)} d \theta \ll (1 + r \cdot d(g, N_{G}(A)))^{-1 / 2} \]
uniformly in $g$ and $r \geq 0$.
\end{corollary}

\begin{proof} Let $U$ be the neighborhood from Lemma~\ref{firstlemmadistance}. Take $g_{0} \in N_{G}(A)$. When $g \in Ug_{0}$, then Lemma~\ref{firstlemmadistance} applied to $g g_{0}^{-1}$ and $b(\theta) = c_{1}(\theta, g)$ shows that
\begin{align*}
\int_{0}^{2 \pi} c_{1}(\theta, g) e^{ir \psi(\theta, g)}
& = \int_{0}^{2 \pi} c_{1}(\theta, g) e^{ir \psi(\theta, g g_{0}^{-1})} d \theta \\
& \ll_{g_{0}} (1 + r \cdot d(g g_{0}^{-1}, A))^{-1 / 2} \\
& \ll_{g_{0}} (1 + r \cdot d(g, N_{G}(A)))^{-1 / 2} 
\end{align*}
uniformly in $g \in Ug_{0}$. We may now take a fixed subset $N_{0} \subset N_{G}(A)$ such that the $U g_{0}$ for $g_{0} \in N_{0}$ form a locally finite cover of $N_{G}(A)$, and let $V = \bigcup_{g_{0} \in N_{0}} U g_{0}$.
\end{proof}

\begin{proof}[Proof of Proposition~\ref{boundorbitalintegral}] By the first part of Proposition~\ref{boundorbitalintegral}, we may restrict to the $g \in G$ with $d(g, e) \leq C'$. The estimates from Lemma~\ref{corputawayfromnormalizer} and Corollary~\ref{corputatnormalizer} may be plugged into Lemma~\ref{reductiontothetaintegral} to obtain
\[ |J(r, g)| \ll r^{-1} (1 + r \cdot d(g, N_{G}(A)))^{-1 / 2} \]
as $r \to \infty$, uniformly when $d(g, e) \leq C'$. On the one hand we have, by estimating $|J(r, g)| \beta(r) \ll 1$, that
\begin{align*}
I(\nu, g)
& = \int_{0}^{\infty} J(r, g) \widehat{k}_{\nu}(i r) \beta(r) dr \\
& \ll \int_{0}^{\infty} \widehat{k}_{\nu}(ir) dr \\
& \ll 1 \,.
\end{align*}
On the other hand, when $g \notin N_{G}(A)$,
\begin{align*}
I(\nu, g) & \ll \int_{0}^{\nu / 2} \widehat{k}_{\nu}(ir) dr + \int_{\nu / 2}^{\infty} (r \cdot d(g, N_{G}(A)))^{-1 / 2} \widehat{k}_{\nu}(ir) dr \\
& \ll_{N} \nu^{-N} + d(g, N_{G}(A))^{-1 / 2} \nu^{-1 / 2} \int_{\nu / 2}^{\infty} \widehat{k}_{\nu}(ir) dr \\
& \ll d(g, N_{G}(A))^{-1 / 2} \nu^{-1 / 2} \,,
\end{align*}
where we have used that $1 \ll d(g, e)^{-1/2} \leq d(g, N_{G}(A))^{-1/2}$. Combining the two estimates, the claim follows.
\end{proof}

\section{Proof of Theorem~\ref{maintheorem}}

\subsection{Amplification}

\label{sectionamplification}

Let $(a(n))_{n \geq 1}$ be a sequence of nonnegative real numbers, supported on a finite set of integers $n \geq 1$ which are coprime to the discriminant $\Delta_{R}$. Let $T = (\sum a_{n} T_{n})^{2}$, and denote by $\widehat{T}(\phi_{j}) \geq 0$ the $T$-eigenvalue of $\phi_{j}$. Define for $\nu \geq 0$,
\begin{align*}
Q(\nu, a) &:= \sum_{j} \widehat{k}_{\nu}(i \nu_{j}) \widehat{T}(\phi_{j}) = B(a) \operatorname{Vol}(\Gamma \backslash \mathfrak{h}) k_{\nu}(e) + O \left( R(a) \right) \, ; \\
Q_{L}(\nu, a) &:= \sum_{j} \widehat{k}_{\nu}(i \nu_{j}) \widehat{T}(\phi_{j}) \mathscr{P}(\phi_{j})^{2} \,.
\end{align*}
Our aim is to obtain good asymptotic estimates for these spectral sums.

\begin{lemma} \label{boundsumorbitalintegral} Define the orbital integral $I(\nu, g)$ by \eqref{definitionorbitalintegral}. There exists an absolute constant $C > 0$, such that for $n \geq 1$ and $\nu \geq 0$ we have
\[ \sum_{\eta \in R(n) - N_{R^{+}}(F)} |I(\nu, g_{0}^{-1} \eta g_{0})| \ll (1 + \nu)^{-1 / 2} n e^{C \log n / \log \log (n + 1)} \,. \]
\end{lemma}

\begin{proof}
By Lemma~\ref{boundorbitalintegral}, there exists $C' > 0$ independent of $\nu$ and $n$ such that only the terms with $d(g_{0}^{-1} \eta g_{0}, e) \leq C'$ have a nonzero contribution. Cover the interval $[0, C']$ with the intervals $I_{0} = [0, (1 + \nu)^{-1}]$, $I_{k} = [e^{k- 1} (1 + \nu)^{-1}, e^{k} (1 + \nu)^{-1}]$ for $1 \leq k \leq \log (C' (1 + \nu)) + 1$. Define $M(n, \delta)$ as in \S \ref{countingalmostheckereturns}, with respect to this value of $C'$. When $d(g_{0}^{-1} \eta g_{0}, N_{G}(A)) \in I_{0}$, we apply the bounds
\begin{align*}
I(\nu, g_{0}^{-1} \eta g_{0}) & \ll 1 \,, \\
\# M(n, (1 + \nu)^{-1}) & \ll (1 + \nu)^{-1} n e^{C \log n / \log \log(n + 1)}
\end{align*}
from Proposition~\ref{boundorbitalintegral} and Lemma~\ref{countalmostheckereturnsA}, which imply that the contribution of all $\eta$ with $d(g_{0}^{-1} \eta g_{0}, N_{G}(A)) \in I_{0}$ is bounded by $\ll (1 + \nu)^{-1} n e^{C \log n / \log \log(n + 1)}$ for some $C > 0$.
When $d(g_{0}^{-1} \eta g_{0}, N_{G}(A)) \in I_{k}$, we apply the bounds
\begin{align*}
I(\nu, g_{0}^{-1} \eta g_{0}) & \ll e^{-k / 2} \,, \\
\# M(n, e^{k} (1 + \nu)^{-1}) & \ll e^{k} (1 + \nu)^{-1} n e^{C \log n / \log \log(n + 1)} \,.
\end{align*}
Summing over $k$, the contribution of the $\eta$ with $d(g_{0}^{-1} \eta g_{0}, N_{G}(A)) \notin I_{0}$ is bounded by $\ll (1 + \nu)^{-1 / 2} n e^{C \log n / \log \log(n + 1)}$.\end{proof}

\begin{proposition} \label{amplifiedpretracewitherror}There exists an absolute constant $C > 0$ such that for $(a_{n})$ as above and for all $\nu \geq 0$,
\begin{align*}
Q(\nu, a) & = B(a) \operatorname{Vol}(\Gamma \backslash \mathfrak{h}) k_{\nu}(e) + O(R(a)) \,, \\
Q_{L}(\nu, a) & = B_{L}(a) \operatorname{Vol}(\Gamma_{L} \backslash L) \int_{\mathbb{R}} k_{\nu}(a(t)) dt + O((1 + \nu)^{-1/2} R_{L}(a)) \,,
\end{align*}
where
\begin{align*}
B(a) &:= \sum_{m, n} a_{m} a_{n} \sum_{d \mid m, n} d \cdot s(mn / d^{2}) \, , \\
R(a) &:= \sum_{m, n} a_{m} a_{n} \sum_{d \mid m, n} d \cdot \left( \frac{mn}{d^{2}} \right)^{3 / 2} \cdot e^{C \log(mn / d^{2}) / \log \log(1 + mn / d^{2})} \,, \\
B_{L}(a) &:= \sum_{m, n} a_{m} a_{n} \sum_{d \mid m, n} d \cdot |N_{R(mn / d^{2})}(F) / R_{F}^{1}| \,, \\
R_{L}(a) &:= \sum_{m, n} a_{m} a_{n} \sum_{d \mid m, n} d \cdot \frac{mn}{d^{2}} \cdot e^{C \log(mn/d^{2}) / \log \log (1 + mn/d^{2})} \, .
\end{align*}
\end{proposition}

\begin{proof}The first asymptotic formula follows from the recurrence relation \eqref{Heckerecurrence} and Proposition~\ref{standardpretraceasymptotic}. The second statement follows similarly from Proposition~\ref{mainpluserrorworkable} and the estimate from Lemma~\ref{boundsumorbitalintegral}.
\end{proof}

In order to bound the tails of the spectral sums $Q(\nu, a)$ and $Q_{L}(\nu, a)$, we shall need:

\begin{lemma} \label{boundsOonesums} Let $(a_{n})$ be as above. Then for $\nu \geq 0$,
\begin{align*}
\sum_{|\nu_{j} - \nu| \leq 1} \widehat{T}(\phi_{j}) & \ll (\nu + 1) \cdot B(a) + R(a) \, , \\
\sum_{|\nu_{j} - \nu| \leq 1} \widehat{T}(\phi_{j}) \mathscr{P}(\phi_{j})^{2} & \ll B_{L}(a) + (\nu + 1)^{-1 / 2} R_{L}(a) \, ,
\end{align*}
uniformly in $(a_{n})$ and $\nu$.
\end{lemma}

\begin{proof} For the first estimate, we start from Proposition~\ref{amplifiedpretracewitherror} and plug in the bound $k_{\nu}(e) \asymp \nu$, to obtain
\[ \sum_{j} \widehat{k}_{\nu}(i \nu_{j}) \widehat{T}(\phi_{j}) \leq C_{0} \cdot \nu \cdot B(a) + R(a) \,, \]
for some $C_{0} > 0$ and for $\nu \geq \nu_{(0)}$, where $\nu_{(0)} > 0$ depends on the choice of the family $(k_{\nu})$. Here, we used that $B(a), R(a) \geq 0$. The terms in the LHS are nonnegative. Because $\widehat{k}_{\nu}(i \nu_{j}) \geq 1$ when $|\nu_{j} - \nu| \leq 1$, the claim follows for $\nu \geq \nu_{(0)}$ by discarding the terms with $|\nu_{j} - \nu| > 1$.

To treat the case where $\nu \leq \nu_{(0)}$, one may either use Proposition~\ref{amplifiedpretracewitherror} and observe that our specific construction of the family $(k_{\nu})$ satisfies $k_{\nu}(e) \ll 1$ for $\nu \ll 1$, or use the following argument: for every $\nu \leq \nu_{(0)}$, we may find integers $n_{1}, n_{2} \in [0, \nu_{(0)}]$ such that the set $\overline{B}(\nu, 1) \cap \{ \nu_{j} : j \geq 0 \}$ is contained in $\overline{B}(n_{1}, 1) \cup \overline{B}(n_{2}, 1)$. Proposition~\ref{amplifiedpretracewitherror} applied to $k_{n_{1}}$ and $k_{n_{2}}$ gives
\begin{align*}
\sum_{|\nu_{j} - \nu| \leq 1} \widehat{T}(\phi_{j})
& \leq \sum_{j} \widehat{k}_{n_{1}}(\nu_{j}) \widehat{T}(\phi_{j}) + \sum_{j} \widehat{k}_{n_{2}}(\nu_{j}) \widehat{T}(\phi_{j}) \\
& \ll \max_{n \in [0, \nu_{(0)}] \cap \mathbb{Z}} k_{n}(e) \cdot B(a) + R(a) \, ,
\end{align*}
as desired.
The second estimate in the statement is proven similarly, by using the bound $\int_{\mathbb{R}} k_{\nu}(a(t)) \ll 1$ from Proposition~\ref{integralmainterm}.
\end{proof}

\begin{proposition} \label{truncatedspectralsumasymptotic} For $(a_{n})$ as above, $C \geq 1$ and $\nu > C$,
\begin{align*}
\sum_{|\nu_{j} - \nu| \leq C} \widehat{k}_{\nu}(i \nu_{j}) \widehat{T}(\phi_{j})
& - B(a) \operatorname{Vol}(\Gamma \backslash \mathfrak{h}) k_{\nu}(e) \\
& \ll C^{-1} \nu B(a) + R(a) \, , \\
\sum_{|\nu_{j} - \nu| \leq C} \widehat{k}_{\nu}(i \nu_{j}) \widehat{T}(\phi_{j}) \mathscr{P}(\phi_{j})^{2}
& - B_{L}(a) \operatorname{Vol}(\Gamma_{L} \backslash L) \int_{\mathbb{R}} k_{\nu}(a(t)) dt \\
& \ll C^{-1} B_{L}(a) + \nu^{-1/2} R_{L}(a) \, ,
\end{align*}
where the implicit constant is uniform in $(a_{n})$, $C$ and $\nu$.
\end{proposition}

\begin{proof}This will follow by combining Lemma~\ref{boundsOonesums} with the rapid decay of $\widehat{k}_{\nu}$. By Proposition~\ref{amplifiedpretracewitherror}, for the first estimate it suffices to prove that
\[ \sum_{|\nu_{j} - \nu| > C} \widehat{k}_{\nu}(i \nu_{j}) \widehat{T}(\phi_{j}) \ll C^{-1} \nu B(a) + R(a) \, . \]
When $\nu > C$, the condition $|\nu_{j} - \nu| > C$ is equivalent to $|\operatorname{Re}(\nu_{j}) - \nu| > C$. Indeed, when $\nu_{j} \in \mathbb{R}$ this is trivial, and when $\nu_{j} \in [-i/2, i/2]$, both conditions are satisfied. Consider first the sum over the set $\{ j : \operatorname{Re}(\nu_{j}) < \nu - C \}$.
We break it up into sums where $\operatorname{Re}(\nu_{j})$ belongs to an interval of length $1$: For $n \geq 0$, we have
\begin{align*}
\sum_{\nu - C - (n + 1) \leq \operatorname{Re}(\nu_{j}) < \nu - C - n} \widehat{k}_{\nu}(i \nu_{j}) \widehat{T}(\phi_{j})
& \ll \frac{1}{(C + n)^{2}} \sum_{|\nu_{j} - (\nu - C - n - \frac{1}{2}) | \leq 1} \widehat{T}(\phi_{j}) \\
& \ll \frac{\nu B(a) + R(a)}{(C + n)^{2}} \,,
\end{align*}
where we use the rapid decay of $\widehat{k}_{\nu}$ and Lemma~\ref{boundsOonesums}.
Summing over integers $n \in [0, \nu - C]$, we find that the sum over $\operatorname{Re}(\nu_{j}) < \nu - C$ is bounded up to a constant by $C^{-1} (\nu B(a) + R(a))$. The sum over $\operatorname{Re}(\nu_{j}) > \nu + C$ is similarly bounded up to a constant by
\begin{align*}
\sum_{n = 0}^{\infty} \frac{(\nu + C + n) B(a) + R(a)}{(C + n)^{3}}
& \leq \sum_{n = 0}^{\infty} \frac{\nu(C + n) B(a) + R(a)}{(C + n)^{3}} \\
& \ll C^{-1} (\nu B(a) + R(a)) \,.
\end{align*}
Combining the two estimates, the first statement follows. For the second statement, it suffices to prove that
\[ \sum_{|\nu_{j} - \nu| > C} \widehat{k}_{\nu}(i \nu_{j}) \widehat{T}(\phi_{j}) \mathscr{P}(\phi_{j})^{2} \ll C^{-1} B_{L}(a) + \nu^{-1/2} R_{L}(a) \, . \]
Using the rapid decay of $\widehat{k}_{\nu}$ and Lemma~\ref{boundsOonesums}, we find that the sum over the $j$ with $\operatorname{Re}(\nu_{j}) > \nu + C$ is bounded up to a constant by
\begin{align*}
\sum_{n = 0}^{\infty} \frac{B_{L}(a) + \nu^{-1/2} R_{L}(a)}{(C + n)^{2}}
& \ll C^{-1} ( B_{L}(a) + \nu^{-1/2} R_{L}(a)) \,.
\end{align*}
The sum over the $j$ with $\operatorname{Re}(\nu_{j}) < \nu - C$ is bounded up to a constant by
\begin{align*}
& \sum_{n \leq \nu - C} \frac{B_{L}(a) + (|\nu - C - n| + 1)^{-1/2} R_{L}(a)}{(C + n)^{2}} \\
& \ll C^{-1} B_{L}(a) + R_{L}(a) \left( \sum_{C + n \leq \sqrt{\nu}} \frac{\nu^{-1/2}}{(C + n)^{2}} + \sum_{\sqrt{\nu} \leq C + n} \frac{1}{(C + n)^{2}} \right) \\
& \ll C^{-1} B_{L}(a) + \nu^{-1/2} R_{L}(a) \, .
\end{align*}
The second statement follows.
\end{proof}

\subsection{Optimal resonators}

\label{sectionoptimization}

Let $F \subset B$ as before be the field corresponding to the geodesic $L$, and let $\mathfrak{f}_{R_{F}}$ be the conductor of $R_{F}$. We seek to apply Proposition~\ref{truncatedspectralsumasymptotic} to a sequence $(a_{n})$ for which the quotient $B_{L}(a) / B(a)$ is large, and for which the sums $R(a)$ and $R_{L}(a)$ are small compared to $B(a)$ resp.\ $B_{L}(a)$. A similar optimization problem has been considered in \cite{milicevic2010} in the context of lower bounds for point evaluations of Maass cusp forms. Recycling some of the results there, we obtain in our situation:

\begin{proposition}[Optimal resonators] \label{optimalresonators} Let $M > 3$ be a real number. There exists a sequence $(a_{n})$ of nonnegative real numbers supported on integers $n \leq M$ coprime to $\Delta_{R}$, such that the following hold:
\begin{enumerate}
\item We have the lower bound \begin{align} \label{maximumquotient}
\frac{B_{L}(a)}{B(a)} \geq \exp \left( 2 \sqrt{2} \sqrt{ \frac{\log M}{\log \log M}} \left( 1 + O \left( \frac{\log \log \log M}{\log \log M} \right) \right) \right) \,.
\end{align}
\item There exists a constant $C'' > 0$ independent of $M$ such that
\begin{align} \label{standardpretraceerrorspecificamplifier}
R(a) & \ll M^{3} e^{C'' \log M / \log \log M} \,, \\
\label{relativepretraceerrorspecificamplifier}
R_{L}(a) & \ll M^{2} e^{C'' \log M / \log \log M} \,.
\end{align}
\item $B(a) \geq 1$.
\end{enumerate}
\end{proposition}

\begin{proof} We first reduce to the precise situation in \cite{milicevic2010}. When $(a_{n})$ is as in \S \ref{sectionamplification}, define the sum $B_{R_{F}}(a)$ by replacing the set $N_{R(mn / d^{2})}(F) / R_{F}^{1}$ by the set $P_{R_{F}}(mn / d^{2})$ in the definition of $B_{L}(a)$. When $(a_{n})$ is supported on integers coprime to $\mathfrak{f}_{R_{F}}$, Lemma~\ref{cardinalitystabilizertau} implies $B_{L}(a) \geq B_{R_{F}}(a)$, so that $B_{L}(a) / B(a) \geq B_{R_{F}}(a) / B(a)$. The latter quantity is the one considered in \cite{milicevic2010}.

When $M$ is sufficiently large, we construct a sequence $(a_{n})$ as follows: Let $L = \sqrt{2 \log M \log \log M}$, and define a multiplicative function $f$ by prescribing its values on prime powers by
\begin{align*}
f(p^{n}) := \begin{cases}
\frac{L}{p \log p} & \text{if } \chi_{F}(p) = 1, \; p \nmid \Delta_{R} \mathfrak{f}_{R_{F}}, \; n = 1, \; L^{2} < p \leq \exp( \log^{2} L) \, , \\
0 & \text{otherwise.}
\end{cases}
\end{align*}
Define $a_{n} = f(n)$ when $n \leq M$ and $a_{n} = 0$ otherwise. The proof of \cite[Lemma 5]{milicevic2010} gives that $B_{R_{F}}(a) / B(a)$ is as large as the RHS in \eqref{maximumquotient}. (And it is shown that this sequence is optimal, in the sense that for every sequence $(a_{n})$ supported on integers coprime to $\Delta_{R} \mathfrak{f}_{R_{F}}$, the ratio $B_{R_{F}} (a) / B(a)$ is at most as large as the RHS in \eqref{maximumquotient}.)

We have $B(a) \geq a_{1}^{2} = 1$. It remains to prove the bounds \eqref{standardpretraceerrorspecificamplifier} and \eqref{relativepretraceerrorspecificamplifier}.
Because $(a_{n})$ is supported on integers $n \leq M$,
\begin{align*}
R(a) & \ll M^{3} e^{C \log(M^{2}) / \log \log (M^{2})} \sum_{m, n} a_{m} a_{n} \sum_{d \mid m, n} \frac{1}{d^{2}} \\
& \ll M^{3} e^{2 C \log M / \log \log M} \left( \sum_{n} a_{n} \right)^{2} \,.
\end{align*}
Using Chebyshev's estimates, we have
\begin{align*}
\sum_{n} a_{n} & \leq \prod_{p} (1 + f(p)) \\
& \leq \exp \left( \sum_{p} f(p) \right) \\
& \leq \exp \left( L \cdot \sum_{L^{2} < p} \frac{1}{p \log p} \right) \\
& \ll \exp \left( L / \log L \right) \\
& \ll \exp \left( \sqrt{\log M} \right) 
\end{align*}
which proves \eqref{standardpretraceerrorspecificamplifier}. Similarly, the estimate
\[ R_{L}(a) \ll M^{2} e^{2 C \log M / \log \log M} \sum_{m, n} a_{m} a_{n} \sum_{d \leq M} \frac{1}{d} \]
gives us \eqref{relativepretraceerrorspecificamplifier}.
\end{proof}

Armed with this resonator sequence, we are ready to prove Theorem~\ref{maintheorem}.

\begin{proof}[Proof of Theorem~\ref{maintheorem}] Let $\nu > 0$ be large. Let $C''$ be the constant from Proposition~\ref{optimalresonators}, choose any $A > C''/8$ and let $M = \nu^{1/4} e^{-A \log \nu / \log \log \nu}$. Let $(a_{n})$ be the corresponding sequence given by Proposition~\ref{optimalresonators}. We check that $R(a)$ and $R_L(a)$ are small compared to $B(a)$ and $B_L(a)$. From \eqref{standardpretraceerrorspecificamplifier} we have $R(a) \ll_{\epsilon} \nu^{3/4 + \epsilon}$, so that $R(a) = o(\nu B(a))$. From \eqref{relativepretraceerrorspecificamplifier} we have
\begin{align*}
R_{L}(a) & \ll \nu^{1/2} e^{-2 A \log \nu / \log \log \nu} e^{C'' \log M / \log \log M} \\
& \ll \nu^{1/2} e^{(-2 A + C''/4) (\log \nu / \log \log \nu) (1 + o(1))} \\
& \ll \nu^{1/2} \,,
\end{align*}
so that $\nu^{-1/2} R_{L}(a) = o(B_{L}(a))$.
By Proposition~\ref{choiceconvolution} and Proposition~\ref{integralmainterm}, there exists $C \geq 1$ such that
\begin{align*}
C^{-1} \nu & \leq \frac{1}{2} \operatorname{Vol}(\Gamma \backslash \mathfrak{h}) k_{\nu}(e) \,, \\
C^{-1} & \leq \frac{1}{2} \operatorname{Vol}(\Gamma_{L} \backslash L) \int_{\mathbb{R}} k_{\nu}(a(t)) dt \,,
\end{align*}
for $\nu$ sufficiently large.
Proposition~\ref{truncatedspectralsumasymptotic} applied to the sequence $(a_{n})$ and such $C$ gives
\begin{align*}
\sum_{|\nu_{j} - \nu| \leq C} \widehat{k}_{\nu}(i \nu_{j}) \widehat{T}(\phi_{j})
& \asymp B(a) \nu \,, \\
\sum_{|\nu_{j} - \nu| \leq C} \widehat{k}_{\nu}(i \nu_{j}) \widehat{T}(\phi_{j}) \mathscr{P}(\phi_{j})^{2}
& \asymp B_{L}(a) \,.
\end{align*}
In particular, there must exist at least one $\nu_{j} \in [\nu - C, \nu + C]$ with $\mathscr{P}(\phi_{j})^{2} \gg B_{L}(a) / (\nu B(a))$. That is,
\begin{align*}
(\nu^{1/2} \mathscr{P}(\phi_{j}))^{2}
& \gg \frac{B_{L}(a)}{B(a)} \\
& \gg \exp \left( 2 \sqrt{2} \sqrt{ \frac{\log M}{\log \log M}} \left( 1 + O \left( \frac{\log \log \log M}{\log \log M} \right) \right) \right) \,.
\end{align*}
We have $\log M = \frac{1}{4} \log \nu \cdot (1 + O(\log \log \nu / \log \nu))$, $\log \log M = \log \log \nu \cdot (1 + O(1 / \log \nu ))$ and $\log \log \log M \ll \log \log \log \nu$. Using this, we obtain for this particular $j$,
\[ (\nu^{1/2} \mathscr{P}(\phi_{j}))^{2} \gg \exp \left( \sqrt{2} \sqrt{ \frac{\log \nu}{\log \log \nu}} \left( 1 + O \left( \frac{\log \log \log \nu}{\log \log \nu} \right) \right) \right) \,. \]
Taking square roots and writing $\nu = \sqrt{\lambda - \frac14}$, the theorem follows.
\end{proof}

\subsection{Extreme values of \texorpdfstring{$L$}{L}-functions}

\label{sectioncorollary}

We now explain how to deduce Corollary~\ref{corollaryLfunction} from (the proof of) Theorem~\ref{maintheorem}. There are two issues: In general, the RHS of the period formula \eqref{popaformula} involves a sum of periods, rather than a single one. Second, Theorem~\ref{maintheorem} does not necessarily produce even $\phi_{j}$.

\begin{proof}[Proof of Corollary~\ref{corollaryLfunction}] We use the notations from the statement. Let $B, R, \Gamma, \Lambda_{d_{F}}$ be as in Remark~\ref{remarkpopaformula}, and associate to $f_j$ the Laplace--Hecke eigenfunction $\phi_j$ on $\Gamma \backslash\mathfrak h$. By our assumptions on $N$ and $d_{F}$, $B$ is a quaternion division algebra and $R$ a maximal order. Let $\phi_0$ be the constant function $\operatorname{Vol}(\Gamma \backslash \mathfrak h)^{-1/2}$. Then $(\phi_{j})_{j \geq 0}$ is an orthonormal basis of $L^{2}(\Gamma \backslash \mathfrak{h})$ as in \S \ref{notationmaassforms}.
In \eqref{popaformula}, take $f = f_{j}$. The product of the Gamma factors in the LHS of \eqref{popaformula} is of size $\lambda_{j}^{-1/2}$ by Stirling's formula,
and we have the lower bound $L(1, \pi_{j}, \operatorname{Ad}) \gg 1/ \log (1 + \lambda_{j})$ \cite{goldfeld1994}. We have the factorization
\[ L(s, \pi_{j} \times \pi_{\chi_0}) = L(s, \pi_{j}) L(s, \pi_{j} \times \omega_{F}) \,.\]
Thus it suffices to give a lower bound for the RHS of \eqref{popaformula}, of the same quality as in Theorem~\ref{maintheorem}.

Let $\sigma \in R$ be an element of norm $-1$. In the amplified pre-trace formula \eqref{amplifiedpretrace}, replace $x$ by $\overline{\sigma x}$ (where the bar denotes complex conjugation) and add the resulting equality to \eqref{amplifiedpretrace} to obtain 
\[ 2 \sum_{\phi_{j} \text{ even}} \widehat{k}_{\nu}(i \nu_{j}) \widehat{T}_{n}(\phi_{j}) \phi_{j}(x) \phi_{j}(y) = \sum_{\eta \in R(n) / \pm 1} \left( k_{\nu}(x^{-1} \eta y) + k_{\nu}((\overline{\sigma x})^{-1} \eta y) \right) \]
For every pair $(\ell, \ell')$ of closed geodesics in $\Lambda_{d_{F}}$, choose lifts $L, L'$ and fundamental domains $\mathscr{F}_{L}, \mathscr{F}_{L'}$ and integrate over $\mathscr{F}_{L} \times \mathscr{F}_{L'}$ to obtain
\begin{align*}
2 \sum_{\phi_{j} \text{ even}} \widehat{k}_{\nu}(i \nu_{j}) \widehat{T}_{n}(\phi_{j}) \mathscr{P}_{\ell}(\phi_{j}) \mathscr{P}_{\ell'}(\phi_{j})
& = \sum_{\eta \in R(n) / \pm 1} \int_{\mathscr{F}_{L} \times \mathscr{F}_{L'}} k_{\nu}(x^{-1} \eta y) dx dy \\
& + \sum_{\eta \in R(n) / \pm 1} \int_{\mathscr{F}_{L} \times \mathscr{F}_{L'}}  k_{\nu}((\overline{\sigma x})^{-1} \eta y) dx dy \,.
\end{align*}
Say $L = g_{0}Ai$ and $L' = g_{0}' A i$, and call the first sum in the RHS $S(g_{0}, g_{0}')$. The second sum is then $S \left( \sigma g_{0} \begin{psmallmatrix}
-1 & 0 \\
0 & 1
\end{psmallmatrix}, g_{0}' \right)$. We discuss the first term first; the second will be dealt with similarly. Let $F, F'$ be the subfields of $B$ associated to $L, L'$. In $S(g_{0}, g_{0}')$, we isolate a main term, which is the sum over $\eta$ with $\eta L' = L$, i.e.\ $\eta F' \eta^{-1} = F$. An unfolding argument as in the proof of Proposition~\ref{mainpluserrorworkable} shows that the main term equals
\[ \# (N_{F, F'}(n) / R_{F'}^{1}) \operatorname{Vol}(\Gamma_{L} \backslash L) \int_{\mathbb{R}} k_{\nu}(a(t)) dt \,, \]
where $N_{F, F'} = \{ \eta \in R^{+} : \eta F' \eta^{-1} = F \}$.
In the error term, we may again introduce nonnegative smooth cutoff functions and write it as
\[ \sum_{\eta \in (R(n) - N_{F, F'}) / \pm 1} I_{F, F'}(\nu, g_{0}^{-1} \eta g_{0}') \,, \]
where $I_{F, F'}(\nu, g_{0})$ is defined similarly to \eqref{definitionorbitalintegral}, in terms of the chosen cutoff functions. This error term may be bounded as in Lemma~\ref{boundsumorbitalintegral}. Here, the diophantine problem consists of counting $\eta \in R(n)$ that are at distance $\ll \sqrt n$ from $0$ and are close to $N_{F, F'}$. Multiplication on the right by a fixed element $\eta_{0} \in N_{F', F}$ gives an injection $N_{F, F'}(n) \to N_{R(n N_{B / \mathbb{Q}}(\eta_{0}))}(F)$, reducing to the counting problem in \S \ref{countingalmostheckereturns}. The analogue of Proposition~\ref{boundorbitalintegral} remains true; the only thing that changes in its proof is the amplitude function in the oscillating integral.

For the sum $S \left( \sigma g_{0} \begin{psmallmatrix}
-1 & 0 \\
0 & 1
\end{psmallmatrix}, g_{0}' \right)$ we do the same, except that $L$ is now replaced by $\overline{\sigma L}$ and $F$ is replaced by $\sigma F \sigma^{-1}$. Summing over all pairs $(\ell, \ell')$ and taking an amplifier $(a_{n})$ as in \S \ref{sectionamplification}, Proposition~\ref{truncatedspectralsumasymptotic} becomes
\begin{align} \label{relativewitherrorsumsofperiods}
\begin{split}
\sum_{\substack{\phi_{j} \text{ even} \\ |\nu_{j} - \nu| \leq C}} \widehat{k}_{\nu}(i \nu_{j}) \widehat{T}(\phi_{j}) \left( \sum_{\ell \in \Lambda_{d_{F}}} \mathscr{P}_{\ell}(\phi_{j}) \right)^{2} & = \operatorname{Vol}(\Gamma_{L} \backslash L) B_{d_{F}}(a) \int_{\mathbb{R}} k_{\nu}(a(t)) \\
& + O( C^{-1} B_{d_{F}}(a) + R_{d_{F}}(a))
\end{split}
\end{align}
where $L$ is a lift of an arbitrary $\ell \in \Lambda_{d_{F}}$ (the volume is independent of the choice), $B_{d_{F}}(a)$ is defined similarly to $B_{L}(a)$ (see Proposition~\ref{amplifiedpretracewitherror}) and $B_{d_{F}}(a) \geq B_{L}(a)$. We conclude by choosing the amplifier as in Proposition~\ref{optimalresonators}, taking $C > 0$ large enough and comparing \eqref{relativewitherrorsumsofperiods} to the standard pre-trace formula. (In the latter, we don't bother to extract the contribution of even $\phi_{j}$, since we can discard the contribution of odd ones using positivity, in the endgame.)
\end{proof}

Note that we have placed ourselves in a situation where the order $R$ is maximal, in order to avoid having to worry about oldforms. When $R$ is Eichler of arbitrary level $N$, it might be possible to extract the contribution of newforms, by writing down the pre-trace formulas for Eichler orders $R_{d}$ of levels $d \mid N$ and using M\"{o}bius-inversion. But the arithmetic of the cardinalities $\# (N_{F, F'}(n) /R_{d}^{1})$ (which appear in the main term in \eqref{relativewitherrorsumsofperiods}) as $d$ varies, is intricate in general.

\bibliographystyle{plain}

\bibliography{bibliography}

\begin{thebibliography}{10}

\bibitem{abert2018}
M.~Abert, N.~Bergeron, and E.~Le~Masson.
\newblock Eigenfunctions and random waves in the {B}enjamini--{S}chramm limit.
\newblock preprint at \url{https://arxiv.org/abs/1810.05601}, 2018.

\bibitem{aurich1991}
R.~Aurich and F.~Steiner.
\newblock Exact theory for the quantum eigenstates of the
  {H}adamard--{G}utzwiller model.
\newblock {\em Physica D}, 48(2--3):445--470, March 1991.

\bibitem{aurich1993}
R.~Aurich and F.~Steiner.
\newblock Statistical properties of highly excited quantum eigenstates of a
  strongly chaotic system.
\newblock {\em Physica D}, 64(1--3):185--214, April 1993.

\bibitem{avakumovic1956}
V.~G. Avakumovi\'{c}.
\newblock \"{U}ber die {E}igenfunktionen auf geschlossenen {R}iemannschen
  {M}annigfaltigkeiten.
\newblock {\em Math. Zeitschr.}, 65:327--344, 1656.

\bibitem{berry1977}
M.~V. Berry.
\newblock Regular and irregular semiclassical wavefunctions.
\newblock {\em J. Phys. A}, 10(12):2083--2091, 1977.

\bibitem{blomer2020}
V.~Blomer, \'{E} Fouvry, E.~Kowalski, Ph. Michel, D.~Mili\'{c}evi\'{c}, and
  W.~Sawin.
\newblock {\em The second moment theory of families of ${L}$-functions}.
\newblock Memoirs of the American Mathematical Society. to appear.

\bibitem{bondarenko2017}
A.~Bondarenko and K.~Seip.
\newblock Large greatest common divisor sums and extreme values of the
  {R}iemann zeta function.
\newblock {\em Duke Math. J.}, 166(9):1685--1701, 2017.

\bibitem{brumley2020}
F.~Brumley and S.~Marshall.
\newblock Lower bounds for {M}aass forms on semisimple groups.
\newblock {\em Compos. Math.}, to appear.

\bibitem{chen2015}
X.~Chen and C.~D. Sogge.
\newblock On integrals of eigenfunctions over geodesics.
\newblock {\em Proc. Am. Math. Soc.}, 143(1):151--161, 2015.

\bibitem{delabreteche2019}
R.~de~la Bret\`{e}che and G.~Tenenbaum.
\newblock Sommes de {G}\'{a}l et applications.
\newblock {\em Proc. London Math. Soc.}, 119(3):104--134, 2019.

\bibitem{dyatlov2013}
S.~Dyatlov and M.~Zworski.
\newblock Quantum ergodicity for restrictions to hypersurfaces.
\newblock {\em Nonlinearity}, 26(1):35--52, 2013.

\bibitem{eichler1965}
M.~Eichler.
\newblock {\em Lectures on modular correspondences}.
\newblock Tata Institute of Fundamental Research, 1965.

\bibitem{farmer2007}
D.~W. Farmer, S.~M. Gonek, and C.~P. Hughes.
\newblock The maximum size of ${L}$-functions.
\newblock {\em J. Reine Angew. Math.}, 609:215--236, August 2007.

\bibitem{goldfeld1994}
D.~Goldfeld, J.~Hoffstein, and D.~Lieman.
\newblock Appendix: An effective zero-free region.
\newblock {\em Ann. Math.}, 140(1):177--181, July 1994.

\bibitem{hejhal1992}
D.~A. Hejhal and B.~N. Rackner.
\newblock On the topography of {M}aass waveforms for {PSL}(2,{Z}).
\newblock {\em Exp. Math.}, 1(4):275--305, 1992.

\bibitem{hilberdink2009}
T.~Hilberdink.
\newblock An arithmetical mapping and applications to $\omega$-results for the
  {R}iemann zeta function.
\newblock {\em Acta. Arith.}, 139(4):341--367, 2009.

\bibitem{hormander1968}
L.~H\"{o}rmander.
\newblock The spectral function of an elliptic operator.
\newblock {\em Acta Math.}, 121:193--218, 1968.

\bibitem{iwaniec2002}
H.~Iwaniec.
\newblock {\em Spectral methods of automorphic forms}, volume~53 of {\em
  Graduate Studies in Mathematics}.
\newblock American Mathematical Society, 2nd edition, 2002.

\bibitem{iwaniec1995}
H.~Iwaniec and P.~Sarnak.
\newblock ${L}^\infty$ norms of eigenfunctions on arithmetic surfaces.
\newblock {\em Ann. Math.}, 141:301--320, 1995.

\bibitem{kahane1985}
J.-P. Kahane.
\newblock {\em Some random series of functions}, volume~5 of {\em Cambridge
  studies in advanced mathematics}.
\newblock Cambridge University Press, 2nd edition, 1985.

\bibitem{marshall2016}
S.~Marshall.
\newblock Geodesic restrictions of arithmetic eigenfunctions.
\newblock {\em Duke Math. J.}, 165(3), 2016.

\bibitem{milicevic2010}
D.~Mili\'{c}evi\'{c}.
\newblock Large values of eigenfunctions on arithmetic hyperbolic surfaces.
\newblock {\em Duke Math. J.}, 155(2):365--401, 2010.

\bibitem{milicevic2011}
D.~Mili\'{c}evi\'{c}.
\newblock Large values of eigenfunctions on arithmetic hyperbolic 3-manifolds.
\newblock {\em Geom. Funct. Anal.}, 21(6):1375--1418, December 2011.

\bibitem{popa2006}
A.~A. Popa.
\newblock Central values of {R}ankin ${L}$-series over real quadratic fields.
\newblock {\em Compos. Math.}, 142(4):811--866, July 2006.

\bibitem{ratner1973}
M.~Ratner.
\newblock The central limit theorem for geodesic flows on $n$-dimensional
  manifolds of negative curvature.
\newblock {\em Isr. J. Math.}, 16:181--197, 1973.

\bibitem{reznikov2015}
A.~Reznikov.
\newblock A uniform bound for geodesic periods of eigenfunctions on hyperbolic
  surfaces.
\newblock {\em Forum Math.}, 27(3):1569--1590, 2015.

\bibitem{rudnick1994}
Z.~Rudnick and P.~Sarnak.
\newblock The behaviour of eigenstates of arithmetic hyperbolic manifolds.
\newblock {\em Commun. Math. Phys.}, 161:195--213, 1994.

\bibitem{salem1954}
R.~Salem and A.~Zygmund.
\newblock Some properties of trigonometric series whose terms have random
  signs.
\newblock {\em Acta Math.}, 91:245--301, 1954.

\bibitem{sands1991}
J.~W. Sands.
\newblock Generalization of a theorem of {S}iegel.
\newblock {\em Acta Arith.}, 58(1):47--56, 1991.

\bibitem{sarnak1995}
P.~Sarnak.
\newblock Arithmetic quantum chaos.
\newblock In {\em The {S}chur lectures}, volume~8 of {\em Israel mathematical
  conference proceedings}, 1995.

\bibitem{sarnak2007}
P.~Sarnak.
\newblock Reciprocal geodesics.
\newblock {\em Clay Math. Proc.}, 7:217--237, 2007.

\bibitem{sinai1960}
Ya.~G. Sinai.
\newblock The central limit theorem for geodesic flows on manifolds of constant
  negative curvature.
\newblock {\em Proc. USSR Acad. Sci.}, 133(6):1303--1306, 1960.
\newblock in Russian.

\bibitem{soundararajan2008}
K.~Soundararajan.
\newblock Extreme values of zeta and ${L}$-functions.
\newblock {\em Math. Ann.}, 342:467--486, 2008.

\bibitem{stein1993}
E.~M. Stein.
\newblock {\em Harmonic analysis: real-variable methods, orthogonality, and
  oscillatory integrals}.
\newblock Princeton University Press, 1993.

\bibitem{toth2012}
J.~A. Toth and S.~Zelditch.
\newblock Quantum ergodic restriction theorems. {I}: interior hypersurfaces in
  domains with ergodic billiards.
\newblock {\em Ann. Henri Poincar\'{e}}, 13(4):599--670, May 2012.

\bibitem{toth2013}
J.~A. Toth and S.~Zelditch.
\newblock Quantum ergodic restriction theorems: manifolds without boundary.
\newblock {\em Geom. Funct. Anal.}, 23(2):715--775, April 2013.

\bibitem{vigneras1980}
M.~Vign\`{e}ras.
\newblock {\em Arithm\'etique des alg\`{e}bres de quaternions}.
\newblock Number 800 in Lecture Notes in Mathematics. Springer-Verlag, 1980.

\bibitem{waldspurger1985}
J.-L. Walspurger.
\newblock Sur les valeurs de certaines fonctions $l$ automorphes et leur centre
  de sym\'{e}trie.
\newblock {\em Compos. Math.}, 54:173--242, 1985.

\bibitem{young2016}
M.~Young.
\newblock The quantum unique ergodicity conjecture for thin sets.
\newblock {\em Adv. Math.}, 286:958--1016, January 2016.

\bibitem{young2018}
M.~Young.
\newblock Equidistribution of {E}isenstein series on geodesic segments.
\newblock {\em Adv. Math.}, 340:1166--1218, 2018.

\bibitem{zelditch1992}
S.~Zelditch.
\newblock Kuznecov sum formulae and {S}zeg{\H{o}} limit formulae on manifolds.
\newblock {\em Comm. Part. Differ. Equat.}, 17(1--2):221--260, 1992.

\bibitem{zelditch2005}
S.~Zelditch.
\newblock Quantum ergodicity and mixing of eigenfunctions.
\newblock In {\em Encyclopedia of mathematical physics}, pages 183--196, 2006.

\bibitem{zhang2001}
S.~Zhang.
\newblock Gross--{Z}agier formula for $\operatorname{GL}_2$.
\newblock {\em Asian J. Math.}, 5(2):183--290, 2001.

\end{thebibliography}

\end{document}